\documentclass[12pt,a4paper,oneside,reqno]{article}
\usepackage[utf8]{inputenc}
\usepackage[english]{babel}
\usepackage[symbol]{footmisc}
\usepackage{amssymb,amsmath,amsthm,amsfonts,xcolor, enumerate,hyperref, comment,longtable, cleveref}
\usepackage[all]{xy}
\usepackage{tikz-cd}
\usetikzlibrary{babel}
\usepackage{tikz}
\usepackage{tkz-graph}
\usepackage{verbatim}
\usepackage{multirow}
\usepackage{times}
\usepackage{cite}
\usepackage{pdflscape}
\usepackage{ulem}
\usepackage[mathcal]{euscript}
\usepackage{tikz}
\usepackage{hyperref}
 \usepackage{mathrsfs}
\usepackage{cancel}
\usepackage{stmaryrd}
\usepackage{amsfonts}
\usepackage{amssymb}
\usepackage{times}
\usepackage{xcolor}
\usepackage{tkz-graph}
\usepackage{url}
\usepackage{float}
\usepackage{tasks}
\usepackage{array}

\usepackage{cite}
\usepackage{hyperref}

 \usepackage{fancyhdr} 
\usepackage{amsfonts}
\usepackage{amsmath}
\usepackage{eurosym}
\usepackage{geometry}

\usepackage{caption,booktabs}

\theoremstyle{plain}
\newtheorem{theorem}{Theorem}
\newtheorem{lemma}[theorem]{Lemma}

\newtheorem{remark}[theorem]{Remark}

\newtheorem{definition}[theorem]{Definition}

\usetikzlibrary{arrows}

\usepackage{fouriernc}

\begin{document}
 
 \bigskip

\noindent{\Large
The geometric classification of left-symmetric \\ 
algebras and superalgebras}

\begin{center}

 {\bf
Kobiljon Abdurasulov\footnote{CMA-UBI, University of  Beira Interior, Covilh\~{a}, Portugal;  \ 
V.I. Romanovskiy Institute of Mathematics, Uzbekistan Academy of Sciences, Tashkent, Uzbekistan; \ abdurasulov0505@mail.ru},
Nigora Daukeyeva\footnote{Chirchik State Pedagogical University, Chirchik, Uzbekistan; \ daukeyeva.cspu@gmail.com}, and 
Azamat Saydaliyev\footnote{V.I. Romanovskiy Institute of Mathematics, Uzbekistan Academy of Sciences, Tashkent, Uzbekistan;  \ azamatsaydaliyev6@gmail.com}
   
}
 
\end{center}

\ 

\noindent {\bf Abstract:}
{\it  
We study the geometric classification of complex left-symmetric algebras and left-symmetric superalgebras in low dimensions. Based on the algebraic classifications of $3$-dimensional left-symmetric algebras (due to Bai, together with the classical classification of $3$-dimensional commutative associative algebras) and of $2$- and $3$-dimensional left-symmetric superalgebras (due to Zhang and Bai), we describe the irreducible components of the corresponding varieties. On the way, we correct some inaccuracies of these algebraic classifications; in particular, we add a one-parameter family of left-symmetric algebras which is missing in the list of Bai and a family of left-symmetric superalgebras which is missing in the list of Zhang and Bai. For the variety of $3$-dimensional left-symmetric algebras, we show that its dimension is $10$ and that it has $19$ irreducible components; six of them are the orbit closures of rigid algebras. For left-symmetric superalgebras of dimensions $(1|1)$, $(1|2)$ and $(2|1)$, we describe the irreducible components, compute their dimensions and identify the rigid superalgebras: the corresponding varieties have dimensions $2$, $5$ and $6$, and they have $3$, $8$ and $10$ irreducible components, respectively; each of them contains exactly one rigid superalgebra.
}

 \bigskip 


\bigskip 

 \
 

	 \bigskip

\ 

\ 


\tableofcontents 
\newpage
{

\section*{Introduction}

Left-symmetric algebras (LSAs) appear naturally in many areas of mathematics and physics. They were initially introduced by A.~Cayley in 1857 in the context of rooted tree algebras \cite{CAY}, but were quickly forgotten until the 1960s. At that time, they reappeared independently in the works of Vinberg \cite{VIN} (in the original Russian version, 1960) and Koszul \cite{KOS} (1961) in the study of convex homogeneous cones and affinely flat manifolds. Independently, they were also introduced by Gerstenhaber \cite{Ger} during the same period. Since then, LSAs have been studied under various names: Vinberg algebras, Koszul algebras, quasi-associative algebras, and (on the opposite side) pre-Lie algebras or Gerstenhaber algebras \cite{CHL}. The commutator of an LSA defines a Lie algebra. Recently, pre-Lie algebras have appeared in the study of stochastic partial differential equations \cite{Foissy2025}, and cohomological aspects have been investigated for left-symmetric Rinehart algebras \cite{BenHassine2024} and for dendriform and pre-Lie algebras \cite{Alhussein2026}.

The algebraic classification of low-dimensional LSAs has been an active area of research. The complete classification of complex $3$-dimensional LSAs was obtained by Bai \cite{Bai}; the LSAs whose commutator is identically zero are exactly the commutative associative algebras, and their classification in dimension $3$ is classical (see, for example, \cite{BaiMeng2001}). Later, Zhang and Bai \cite{ZhangBai2012} classified low-dimensional left-symmetric superalgebras (LSSAs), including all $2$- and $3$-dimensional superalgebras. These algebraic lists provide the foundation for the geometric study of the corresponding varieties. Structural results on related classes, such as free metabelian Novikov and Lie-admissible algebras, have also been obtained \cite{Dauletiyarova2025}.

The geometric approach to varieties of algebras, initiated in the 1970s, studies degenerations and irreducible components of the affine variety defined by a set of polynomial identities. Grunewald and O'Halloran \cite{GRH} pioneered this method for nilpotent Lie algebras, and Burde and Steinhoff \cite{bs} constructed the degeneration graphs for $3$- and $4$-dimensional Lie algebras. Similar investigations have been carried out for many other varieties, including Jordan algebras, Poisson algebras, and Leibniz algebras (see, e.g., \cite{AFK25, afm, ahk, als, FKS1, FKS2}). The geometric classification of varieties of algebras has been surveyed comprehensively in \cite{AKK}, and general classification and structure theory of non-associative algebras are presented in \cite{Kaygorodov2024}; see also \cite{Lopes2024} for a broader perspective on noncommutative algebra and representation theory.

The present paper is devoted to the geometric classification of $3$-dimensional LSAs and of low-dimensional LSSAs over the complex numbers. Using the algebraic classifications of Bai \cite{Bai} and Zhang and Bai \cite{ZhangBai2012} (in which we correct some inaccuracies, see Remarks~\ref{remD} and~\ref{remS}), we determine the irreducible components of the corresponding varieties, compute their dimensions, and find all rigid algebras. 

For the variety of $3$-dimensional LSAs, we prove that its dimension is $10$ and that it has $19$ irreducible components (Theorem~\ref{thm5}): two of them have dimension $10$ and are the closures of $2$-parameter families of algebras, and the remaining seventeen have dimension $9$; six of these seventeen components are orbit closures of rigid algebras, and the other eleven are closures of $1$- or $2$-parameter families of algebras. For LSSAs, we obtain the following results:
\begin{itemize}
\item The variety of $(1|1)$-dimensional LSSAs (i.e., $2$-dimensional LSSAs with one-dimensional even part) is the union of three planes; hence it has dimension $2$ and $3$ irreducible components, exactly one of which is the orbit closure of a rigid superalgebra (Theorem~\ref{thm13}).
\item The variety of $(1|2)$-dimensional LSSAs ($3$-dimensional LSSAs with one-dimensional even part) has dimension $5$ and $8$ irreducible components, exactly one of which is the orbit closure of a rigid superalgebra (Theorem~\ref{thm14}).
\item The variety of $(2|1)$-dimensional LSSAs ($3$-dimensional LSSAs with two-dimensional even part) has dimension $6$ and $10$ irreducible components, exactly one of which is the orbit closure of a rigid superalgebra (Theorem~\ref{thm15}).
\end{itemize}

The paper is organized as follows. Section~1 recalls the necessary notation and definitions, including the concepts of degenerations and irreducible components, and the method that we use to prove degenerations and non-degenerations. Section~2 presents the algebraic classification of $3$-dimensional LSAs and then provides the geometric classification of the corresponding variety. Section~3 presents the algebraic classification of LSSAs (due to Zhang and Bai) and then provides the geometric classification of the varieties of $(1|1)$-, $(1|2)$- and $(2|1)$-dimensional LSSAs.

}
\section{Definitions and notation}

Let $S$ be a non-associative algebra with a bilinear product $(x, y) \mapsto xy$. Then $S$ is called a superalgebra if the underlying vector space of $S$ is $\mathbb{Z}_2$-graded, that is, $S = S_{\bar{0}} \oplus S_{\bar{1}}$ and $S_{\bar\alpha}  S_{\bar\beta} \subseteq S_{\bar\alpha+\bar\beta}$ for $\bar\alpha, \bar\beta \in \mathbb{Z}_{2}$. The elements in $S_{\bar{0}}$ are said to be even and those in $S_{\bar{1}}$ are odd; for a homogeneous element $x\in S_{\bar\alpha}$ we write $|x|=\alpha\in\{0,1\}$.

Let \( V = V_{\bar 0} \oplus V_{\bar 1} \) be a \( \mathbb{Z}_2 \)-graded vector space with a fixed homogeneous basis  
\( \big\{e_1, \ldots, e_m, f_1, \ldots, f_n \big\}\), where $e_1,\ldots,e_m\in V_{\bar 0}$ and $f_1,\ldots,f_n\in V_{\bar 1}$; we say that $V$ has dimension $(m|n)$. 
A  superalgebra structure on \(V\) can be described via structure constants 
\((\alpha_{ij}^k, \beta_{ij}^k, \gamma_{ij}^k, \delta_{ij}^k) \in \mathbb{C}^{m^3+3mn^2}\), where the multiplication is defined as:
\[
e_i e_j = \sum_{k=1}^{m} \alpha_{ij}^k e_k, \quad 
e_i f_j = \sum_{k=1}^{n} \beta_{ij}^k f_k, \quad 
f_i e_j = \sum_{k=1}^{n} \gamma_{ij}^k f_k, \quad 
f_i f_j = \sum_{k=1}^{m} \delta_{ij}^k e_k.
\]
Let $\mathcal{S}^{m|n}$ denote the set of all superalgebras of dimension $(m|n)$ satisfying a family of polynomial superidentities $T$, regarded as a subset $\mathbb{L}(T)$ of the affine variety $\operatorname{Hom}(V \otimes V, V)_{\bar 0}$ of even bilinear maps. Then $\mathcal{S}^{m|n}$ is a Zariski-closed subset of this variety. Algebras (without grading) correspond to the case $n=0$.

The group $G = (\operatorname{Aut} V)_{\bar 0} \simeq \operatorname{GL}(V_{\bar 0}) \times \operatorname{GL}(V_{\bar 1})$ acts on $\mathcal{S}^{m|n}$ by conjugation:
\[
(g * \mu)(x \otimes y) = g \mu(g^{-1} x \otimes g^{-1} y),
\]
for all $x, y \in V$, $\mu \in \mathbb{L}(T)$, and $g \in G$.

Let $\mathcal{O}(\mu)$ denote the orbit of $\mu \in \mathbb{L}(T)$ under the action of $G$, and let $\overline{\mathcal{O}(\mu)}$ be the Zariski closure of $\mathcal{O}(\mu)$. Suppose $J, J' \in \mathcal{S}^{m|n}$ are represented by $\lambda, \mu \in \mathbb{L}(T)$, respectively. We say that $\lambda$ degenerates to $\mu$, denoted $\lambda \to \mu$, if $\mu \in \overline{\mathcal{O}(\lambda)}$. In this case, we have $\overline{\mathcal{O}(\mu)} \subset \overline{\mathcal{O}(\lambda)}$. Therefore, the notion of degeneration does not depend on the particular representatives, and we write $J \to J'$ instead of $\lambda \to \mu$, and $\mathcal{O}(J)$ instead of $\mathcal{O}(\lambda)$. 
We write $J \not\to J'$ to indicate that $J' \notin \overline{\mathcal{O}(J)}$.

If $J$ is represented by $\lambda \in \mathbb{L}(T)$, we say that $J$ is \textit{rigid} in $\mathbb{L}(T)$ if $\mathcal{O}(\lambda)$ is an open subset of $\mathbb{L}(T)$. A subset of a variety is called \textit{irreducible} if it cannot be written as a union of two proper closed subsets. A maximal irreducible closed subset is called an \textit{irreducible component}. In particular, $J$ is rigid in $\mathcal{S}^{m|n}$ if and only if $\overline{\mathcal{O}(\lambda)}$ is an irreducible component of $\mathbb{L}(T)$. It is a well-known fact that every affine variety admits a unique decomposition into finitely many irreducible components.

\medskip 

\noindent{\bf Degenerations.} To find degenerations, we use the standard method described in \cite{GRH,als,ahk,bs}. Let $E_1(t),\ldots,E_{m+n}(t)$ be homogeneous elements of $J$ depending on a parameter $t$ (with $E_1(t),\ldots,E_m(t)\in V_{\bar 0}$ and $E_{m+1}(t),\ldots,E_{m+n}(t)\in V_{\bar 1}$) which form a basis of $V$ for all sufficiently small $t\neq 0$. If the structure constants of $J$ in the basis $E_1(t),\ldots,E_{m+n}(t)$ tend, as $t\to 0$, to the structure constants of $J'$ in the fixed basis, then $J\to J'$. We denote such a degeneration by $J\xrightarrow{(E_1(t),\ldots,E_{m+n}(t))}J'$.

\medskip 

\noindent{\bf Families.} Many of the algebras considered below form families $J(\alpha)$ whose structure constants are rational functions of parameters $\alpha=(\alpha_1,\ldots,\alpha_s)$ ranging over a Zariski dense subset $P$ of $\mathbb{C}^s$. In this case, as usual, we write $\mathcal{O}(J(\alpha))$ for the union $\bigcup_{\alpha\in P}\mathcal{O}(J(\alpha))$; its Zariski closure $\overline{\mathcal{O}(J(\alpha))}$ is an irreducible closed subset of $\mathbb{L}(T)$. Since the structure constants depend rationally on $\alpha$, this closure does not change if $P$ is replaced by any other Zariski dense subset of the set of parameters for which the multiplication table is defined; in particular, it contains every member $J(\alpha_0)$ of the family obtained for an ``exceptional'' value $\alpha_0$ of the parameters (for example, a value excluded in the algebraic classification in order to avoid repetitions). Moreover, if $J(\alpha(t))\xrightarrow{(E_1(t),\ldots,E_{m+n}(t))}J'$ for some rational curve $\alpha(t)$ in the space of parameters, then $J'\in\overline{\mathcal{O}(J(\alpha))}$. We use these remarks without further comment.

The dimension of an orbit is computed by the formula $\dim \mathcal{O}(\mu)=\dim G-\dim \mathfrak{Der}_{\bar 0}(\mu)$, where $\mathfrak{Der}_{\bar 0}(\mu)$ is the space of even derivations of $\mu$. The dimension of $\overline{\mathcal{O}(J(\alpha))}$ is equal to the generic rank of the differential of the map $G\times P\to\mathbb{L}(T)$, $(g,\alpha)\mapsto g*J(\alpha)$; at a point $(1,\alpha)$ the image of this differential is spanned by the tangent space $\{D\cdot J(\alpha)\}$ to the orbit $\mathcal{O}(J(\alpha))$ (here $D$ runs over the even linear maps of $V$ and $(D\cdot\mu)(x,y)=D\mu(x,y)-\mu(Dx,y)-\mu(x,Dy)$) together with the partial derivatives $\partial J(\alpha)/\partial\alpha_i$.

\medskip 

\noindent{\bf Irreducible components and non-degenerations.} 
Let $J(\alpha)$, $\alpha\in P$, be a family of superalgebras in $\mathbb{L}(T)$ and let $J'$ be a superalgebra represented by $\mu\in\mathbb{L}(T)$. We say that the family $J(\alpha)$ \textit{degenerates} to $J'$, and write $J(\alpha)\to J'$, if $\mu\in\overline{\mathcal{O}(J(\alpha))}$; otherwise we call it a \textit{non-degeneration} and write $J(\alpha)\not\to J'$. 
Observe that $\overline{\mathcal{O}(J')}$ is an irreducible component of $\mathbb{L}(T)$ if and only if $J\not\to J'$ for every superalgebra $J\not\cong J'$ and $J(\alpha)\not\to J'$ for every parametric family $J(\alpha)$ which does not contain a superalgebra isomorphic to $J'$. Similarly, the closure $\overline{\mathcal{O}(J(\alpha))}$ of a parametric family is an irreducible component of $\mathbb{L}(T)$ if and only if it is not properly contained in the orbit closure of a superalgebra or in the closure of another parametric family. We use the following ideas to find the irreducible components.

Firstly, if $J\to J'$ and $J\not\cong J'$, then $\dim\mathcal{O}(J)>\dim\mathcal{O}(J')$; more generally, if $\mathcal{C}\subseteq\mathcal{C}'$ are two irreducible closed subsets of $\mathbb{L}(T)$ of the same dimension, then $\mathcal{C}=\mathcal{C}'$.

Secondly, degenerations (of superalgebras and of families) are proved by explicit parametrized bases, as described above.

Thirdly, to prove non-degenerations, we use the following lemmas, see \cite{AFK25, afm} and the references therein (see also \cite{GRH, bs}). Here a \textit{triangular subgroup} of $G=\mathrm{GL}(V_{\bar 0})\times\mathrm{GL}(V_{\bar 1})$ is the subgroup of the pairs of invertible lower triangular matrices (or the subgroup of the pairs of invertible upper triangular matrices) with respect to the fixed homogeneous basis $e_1,\ldots,e_m,f_1,\ldots,f_n$; in the non-graded case ($n=0$) these are the groups of invertible lower (upper) triangular matrices.

\begin{lemma}\label{main1}
Consider two superalgebras $J$ and $J'$ represented by structures in $\mathbb{L}(T)$, and suppose $J\to J'$. Let $\mathcal{W}$ be a Zariski closed set in $\mathbb{L}(T)$ that is stable under the action of a triangular subgroup of $G$. If there is a representation $\lambda$ of $J$ in $\mathcal{W}$, then there is a representation $\mu$ of $J'$ in $\mathcal{W}$.
\end{lemma}

\begin{lemma}\label{main2}
Consider a family of superalgebras $J(\alpha)$, $\alpha\in P$, and a superalgebra $J'$ represented by structures in $\mathbb{L}(T)$, and suppose $J(\alpha)\to J'$. Let $\mathcal{W}$ be a Zariski closed set in $\mathbb{L}(T)$ that is stable under the action of a triangular subgroup of $G$. If there is a representation $\lambda(\alpha)$ of $J(\alpha)$ in $\mathcal{W}$ for every $\alpha\in P$, then there is a representation $\mu$ of $J'$ in $\mathcal{W}$.
\end{lemma}

\noindent Both lemmas follow from the fact that, for a closed subset $\mathcal{W}$ stable under a triangular subgroup $B$, the set $G*\mathcal{W}=\{g*\nu : g\in G,\ \nu\in\mathcal{W}\}$ of all structures having a representation in $\mathcal{W}$ is closed, because $G/B$ is a complete variety. In order to apply Lemmas~\ref{main1} and~\ref{main2}, we give in each case the explicit definition of the appropriate closed set $\mathcal{W}$ in terms of the structure constants $c_{ij}^k$ (with respect to the basis indicated in each case): the set $\mathcal{W}$ consists of all structures whose structure constants satisfy the listed relations. Note that a set $\mathcal{W}$ defined by the vanishing of some structure constants $c_{ij}^k$ is stable under the lower triangular subgroup as soon as, together with every $c_{ij}^k$, it contains all $c_{i'j'}^{k'}$ with $i'\geq i$, $j'\geq j$, $k'\leq k$ (for the upper triangular subgroup, the inequalities are reversed); relations which are invariant under the whole group $G$ (for example, polynomial identities) can of course always be used.

Finally, if $\mathcal{C}_1,\ldots,\mathcal{C}_r$ are irreducible closed subsets of $\mathbb{L}(T)$ such that every superalgebra of the algebraic classification belongs to $\mathcal{C}_1\cup\cdots\cup\mathcal{C}_r$, then every irreducible component of $\mathbb{L}(T)$ is one of the sets $\mathcal{C}_i$, and $\mathcal{C}_i$ is an irreducible component if and only if $\mathcal{C}_i\not\subseteq\mathcal{C}_j$ for all $j\neq i$. If the sets $\mathcal{C}_i$ are pairwise distinct, then $\mathcal{C}_i\not\subseteq\mathcal{C}_j$ holds automatically whenever $\dim\mathcal{C}_i\geq\dim\mathcal{C}_j$; hence it is sufficient to prove that none of the superalgebras (families) defining the sets $\mathcal{C}_i$ belongs to a set $\mathcal{C}_j$ of bigger dimension. When $\mathcal{C}_j$ is the closure of a family $J(\alpha)$ and $\mathcal{C}_i$ is the orbit closure of $J'$ (respectively, the closure of a family $J'(\beta)$), this means $J(\alpha)\not\to J'$ (respectively, $J(\alpha)\not\to J'(\beta)$ for some, equivalently for generic, values of $\beta$), and these non-degenerations are proved by Lemma~\ref{main2}.

All the algebras considered below are over $\mathbb{C}$, and all linear maps are $\mathbb{C}$-linear. For simplicity, whenever we write the multiplication table of an algebra, the products of basis elements whose values are zero are omitted.

\section{Left-symmetric algebras}

\subsection{Algebraic classification}
\begin{definition} 
Let $A$ be a vector space with 
 a bilinear product $(x,y) \mapsto xy$. $A$ is called a left-symmetric algebra if it satisfies 

\begin{longtable}{c}
$(xy)z-x(yz)=(yx)z-y(xz)$ \quad for all $x,y,z\in A$.
\end{longtable} 
\end{definition}

Let us present the results of the algebraic classification that will be used in our work.

First, recall that, by Jacobson \cite{Jac}, up to isomorphism the three-dimensional solvable complex Lie algebras are the following ones (the only other three-dimensional complex Lie algebra, $\mathfrak{sl}_2$, is simple and does not admit compatible left-symmetric structures):

(a) the abelian Lie algebra $\mathbb{C}^3$;

(b) $\mathrm{ H}: \ \ [e_1,e_2]=e_3$;

(c) $\mathrm{ N}: \ \ [e_3,e_2]=e_2$;

(d) $\mathrm{ D}_\ell: \ \ [e_3,e_1]=e_1, \ [e_3,e_2]=\ell e_2,\qquad
0<|\ell|<1$ or $\ell=e^{i\theta}, 0\leq \theta\leq \pi$;

(e) $\mathrm{ E}: \ \ [e_3,e_1]=e_1, \ [e_3, e_2]=e_1+e_2$.

\noindent Note that $\mathrm{D}_\ell\cong\mathrm{D}_{1/\ell}$ for $\ell\neq0$ (an isomorphism is given by $e_1\mapsto e_2$, $e_2\mapsto e_1$, $e_3\mapsto \ell^{-1} e_3$), which explains the restriction on $\ell$ in (d). The values $\ell=1$ ($\theta=0$) and $\ell=-1$ ($\theta=\pi$) are included in (d); in the diagram below the Lie algebras $\mathrm{D}_{1}$ and $\mathrm{D}_{-1}$ are drawn separately from the generic $\mathrm{D}_\ell$ ($\ell\neq\pm1$) because their degenerations are different. Formally, $\mathrm{D}_0$ is the Lie algebra $\mathrm{N}$ (up to renumbering of the basis: $e_1\leftrightarrow e_2$). For this reason the families of left-symmetric algebras on $\mathrm{D}_\ell$ in the table below are also allowed to take the value $\ell=0$; their members with $\ell=0$ are left-symmetric algebras on $\mathrm{D}_0\cong\mathrm{N}$ (see Remark~\ref{remD}).

\begin{theorem}[see \cite{bs}] \label{thm3}
The degenerations of three-dimensional Lie algebras are given in the following diagram (all degenerations are obtained from the drawn arrows by transitivity; the numbers on the left are the dimensions of the orbits, and of the family in the case of $\mathrm{D}_\ell$):

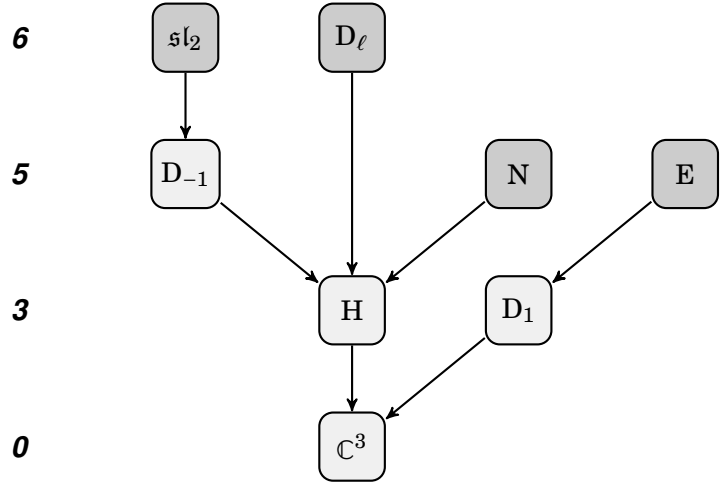
\begin{center}
   \begin{tikzpicture}[->,>=stealth',shorten >=0.0cm,auto,thick,
    every node/.style={font=\sffamily\small\bfseries}]

\newlength{\refwidth}
\settowidth{\refwidth}{$D_{-1}$}
\setlength{\refwidth}{1.2\refwidth}   

\tikzset{
    source node/.style={rectangle,draw,fill=black!20,rounded corners=1ex,
                        minimum width=\refwidth, minimum height=0.9cm,
                        font=\sffamily\small\bfseries},
    target node/.style={rectangle,draw,fill=gray!12,rounded corners=1ex,
                        minimum width=\refwidth, minimum height=0.9cm,
                        font=\sffamily\small\bfseries}
}

\def\colDist{2.2}   

\node (l6) at (-2.2,0)   {\small 6};
\node (l5) at (-2.2,-1.8) {\small 5};
\node (l3) at (-2.2,-3.6) {\small 3};
\node (l0) at (-2.2,-5.4) {\small 0};

\node[source node] (sl2)  at (0,0)              {$\mathfrak{sl}_2$};
\node[source node] (Dl)   at (\colDist,0)       {$\mathrm{D}_\ell$};

\node[target node] (Dm1)  at (0,-1.8)           {$\mathrm{D}_{-1}$};
\node[source node] (N)    at (2*\colDist,-1.8)  {$\mathrm{N}$};
\node[source node] (E)    at (3*\colDist,-1.8)  {$\mathrm{E}$};

\node[target node] (H)    at (\colDist,-3.6)    {$\mathrm{H}$};
\node[target node] (D1)   at (2*\colDist,-3.6)  {$\mathrm{D}_1$};

\node[target node] (C3)   at (\colDist,-5.4)    {$\mathbb{C}^3$};

\draw[->] (Dl)   -- (H);      
\draw[->] (N)    -- (H);      
\draw[->] (sl2)  -- (Dm1);    
\draw[->] (H)    -- (C3);     
\draw[->] (E)    -- (D1);     
\draw[->] (Dm1)  -- (H);      
\draw[->] (D1)   -- (C3);     
\draw[->] (E)    -- (H);      

\end{tikzpicture}
\end{center}

\end{theorem}

 Using these Lie algebras, the algebraic classification of three-dimensional complex left-symmetric algebras with non-abelian commutator was obtained in \cite{Bai}. Since the original classification contains algebras that are special cases of broader parametric families, we have omitted these redundancies and present only representatives of the non-isomorphic families. Consequently, the table below has been revised and renumbered relative to the list in \cite{Bai} (see Remark~\ref{remD} below for the correspondence with the notation of \cite{Bai} and for some comments). Moreover, the following corrections to the list of \cite{Bai} are taken into account (see Remark~\ref{remD}): (i) the algebras (N-28) and (N-29) of \cite{Bai} are the members $\mu=0$ and $\mu=1$ of a one-parameter family $\mathrm{N}_{13}^{\mu}$ of pairwise non-isomorphic algebras, so the members with $\mu\neq0,1$ are added; (ii) the left-symmetric algebras on $\mathrm{N}$ which are obtained in \cite[Remark~3]{Bai} by putting $\ell=0$ in the families on $\mathrm{D}_\ell$ are not listed separately: instead, the value $\ell=0$ is allowed in these families; (iii) in the families on $\mathrm{D}_\ell$ which give repetitions for $\ell=1$ or $\ell=-1$ these values are excluded; (iv) two misprints of \cite{Bai} are corrected (in $\mathrm{D}_{21}$ and $\mathrm{D}_{25}$). With these corrections, every $3$-dimensional complex LSA with non-zero commutator is isomorphic to an algebra of the table, and the algebras of the table are pairwise non-isomorphic, apart from the identifications of parameters indicated in Remark~\ref{remD}. An LSA whose commutator is zero is a commutative associative algebra: indeed, if $xy=yx$ for all $x,y$, then the left-symmetric identity gives $x(yz)=y(xz)$, hence $(xy)z=z(xy)=z(yx)=y(zx)=y(xz)=x(yz)$. The $3$-dimensional complex commutative associative algebras are well known (see, for example, \cite{BaiMeng2001}); they are the algebras $\mathrm{T}_1,\ldots,\mathrm{T}_{12}$ in the table below.

 \begin{longtable}{l l l l l l l}
\hline
\multicolumn{7}{c}{\textbf{The left-symmetric algebras on the abelian Lie algebra $\mathbb{C}^3$}} \\
\multicolumn{7}{c}{\textbf{(commutative associative algebras):}} \\
\hline
$\mathrm{ T}_{1}$ & $e_{1}e_{1}=e_{1}$ & $e_{2}e_{2}=e_{2}$ & $e_{3}e_{3}=e_{3}$ & & & \\ \hline
$\mathrm{ T}_{2}$ & $e_{1}e_{1}=e_{1}$ & $e_{2}e_{2}=e_{2}$ & & & & \\ \hline
$\mathrm{ T}_{3}$ & $e_{1}e_{1}=e_{1}$ & $e_{2}e_{2}=e_{2}$ & $e_{2}e_{3}=e_{3}$ & & & \\
 & $e_{3}e_{2}=e_{3}$ & & & & & \\ \hline
$\mathrm{ T}_{4}$ & $e_{1}e_{1}=e_{1}$ & $e_{2}e_{2}=e_{3}$ & & & & \\ \hline
$\mathrm{ T}_{5}$ & $e_{1}e_{1}=e_{1}$ & $e_{1}e_{2}=e_{2}$ & $e_{2}e_{1}=e_{2}$ & & & \\ \hline
$\mathrm{ T}_{6}$ & $e_{1}e_{1}=e_{1}$ & $e_{1}e_{2}=e_{2}$ & $e_{1}e_{3}=e_{3}$ & & & \\
 & $e_{2}e_{1}=e_{2}$ & $e_{2}e_{2}=e_{3}$ & $e_{3}e_{1}=e_{3}$ & & & \\ \hline
$\mathrm{ T}_{7}$ & $e_{1}e_{1}=e_{1}$ & $e_{1}e_{2}=e_{2}$ & $e_{1}e_{3}=e_{3}$ & & & \\
 & $e_{2}e_{1}=e_{2}$ & $e_{3}e_{1}=e_{3}$ & & & & \\ \hline
$\mathrm{ T}_{8}$ & $e_{1}e_{1}=e_{1}$ & & & & & \\ \hline
$\mathrm{ T}_{9}$ & $e_{1}e_{1}=e_{2}$ & $e_{1}e_{2}=e_{3}$ & $e_{2}e_{1}=e_{3}$ & & & \\ \hline
$\mathrm{ T}_{10}$ & $e_{1}e_{1}=e_{3}$ & $e_{2}e_{2}=e_{3}$ & & & & \\ \hline
$\mathrm{ T}_{11}$ & $e_{1}e_{1}=e_{2}$ & & & & & \\ \hline
$\mathrm{ T}_{12}$ & \multicolumn{6}{l}{trivial algebra} \\ \hline

\multicolumn{7}{c}{\textbf{The left-symmetric algebras on $\rm H$:}} \\
\hline
 $\mathrm{ H}_{1}$  &
 $e_{1}  e_{1}=e_{1}$ &
 $e_{1}  e_{2}=e_{2}+e_{3}$ &
 $e_{1}  e_{3}=e_{3}$\\
 &$e_{2}  e_{1}=e_{2}$&
 $e_{3}  e_{1}=e_{3}$&
\\ \hline

$\mathrm{ H}_{2}$  &
 $e_{1}  e_{1}=e_{1}$ &
 $e_{1}  e_{2}=e_{2}+e_{3}$ &
 $e_{1}  e_{3}=e_{3}$\\
& $e_{2}  e_{1}=e_{2}$&
 $e_{2}  e_{2}=e_{3}$&
 $e_{3}  e_{1}=e_{3}$& 
 \\ \hline

 $\mathrm{ H}_{3}$  &
 $e_{1}  e_{1}=e_{1}$ &
 $e_{1}  e_{2}=e_{3}$ &
 $e_{2}  e_{2}=e_{3}$&
 &
 &
 \\ \hline

 $\mathrm{ H}_{4}$  &
 $e_{1}  e_{1}=e_{1}$ &
 $e_{1}  e_{2}=e_{3}$&
 &
 &
 &
 \\ \hline

$\mathrm{ H}_{5}$  &
 $e_{2}  e_{1}=-e_{3}$ &
 $e_{2}  e_{2}=e_{1}$ &
 &
 &
 &
 \\ \hline

  $\mathrm{ H}_{6}^{\lambda\neq 0}$  &
 $e_{1}  e_{1}=e_{3}$ &
 $e_{1}  e_{2}=e_{3}$ &
$e_{2}  e_{2}=\lambda e_{3}$ &
 &
 &
 \\ \hline

 $\mathrm{ H}_{7}$  &
 $e_{1}  e_{2}=\frac{1}{2}e_{3}$ &
 $e_{2} e_{1}=-\frac{1}{2}e_{3}$ &
 &
 &
 &
 \\ \hline

 $\mathrm{ H}_{8}$  &
 $e_{1}  e_{2}=e_{3}$ &
 $e_{2} e_{2}=e_{1}$ &
 &
 &
 &
 \\ \hline

 $\mathrm{ H}_{9}^{\lambda\neq 1}$  &
 $e_{1}  e_{2}=\frac{\lambda}{\lambda-1}e_{3}$ &
 $e_{2} e_{1}=\frac{1}{\lambda-1}e_{3}$ &
 $ e_{2}  e_{2}=\lambda e_{1}$ &
 &
 &
 \\ \hline

\multicolumn{7}{c}{\textbf{The left-symmetric algebras on $\rm N$:}} \\
\multicolumn{7}{c}{(the other LSAs on $\mathrm{N}$ are the members with $\ell=0$ of the families on $\mathrm{D}_\ell$, see below)} \\
\hline
 
{$\mathrm{ N}_{1}^{\lambda, \mu\neq0}$}  &
  $e_{1}  e_{1}=e_{1}+\frac{\lambda(\lambda-1)}{\mu} e_{3}$ &
 $e_{1}  e_{3}=\lambda e_{3}$ &
 $e_{3}  e_{1}=\lambda e_{3}$&
 & &
 \\

 &
 $e_{3} e_{2}=e_{2}$ &
 $e_{3} e_{3}=\mu e_{3}$&
 & &
 & \\ \hline
 
 $\mathrm{ N}_{2}^\lambda$  &
 $e_{1}  e_{1}=e_{1}+\lambda e_{3}$ &
 $e_{3}  e_{2}=e_{2}$&
 & &
 & \\ \hline

  $\mathrm{ N}_{3}$  &
 $e_{1}  e_{1}=e_{1}$ &
 $e_{3}  e_{2}=e_{2}$ &
 $e_{3}  e_{3}=e_{2}+e_{3}$ &
 & &
 \\ \hline

{$\mathrm{ N}_{4}$}  &
 $e_{1}  e_{3}=e_{2}$ &
 $e_{3} e_{1}=e_{2}$ &
 $e_{3} e_{2}=e_{2}$&
 & &
\\

 &
 $e_{3} e_{3}=-e_{2}+e_{3}$&
 &
 &
 & &
 \\ \hline

{$\mathrm{ N}_{5}^{\lambda}$}  &
 $e_{1}  e_{1}=e_{1}-e_{2}$ &
 $e_{1} e_{3}=e_{2}$ &
 $e_{3} e_{1}=e_{2}$&
 & &
 \\

 &
 $e_{3} e_{2}=e_{2}$&
 $e_{3} e_{3}=-{\lambda e_{2}}+e_{3}$ &
 &
 & &
 \\ \hline

{$\mathrm{ N}_{6}^{\lambda}$}  &
 $e_{1}  e_{1}=e_{1}+{\lambda}e_{3}$ &
 $e_{1} e_{3}=e_{3}$ &
 $e_{3} e_{1}=e_{3}$&
 & &
 \\

 &
 $e_{3} e_{2}=e_{2}$&
 &
 &
 & &
 \\ \hline

 {$\mathrm{ N}_{7}$}  &
 $e_{1}  e_{1}=e_{2}$ &
 $e_{1} e_{3}=e_{1}$ &
 $e_{3} e_{1}=e_{1}$&
 & &
 \\

 &
 $e_{3} e_{2}=e_{2}$ & 
 $e_{3} e_{3}=e_{2}+e_{3}$&
 &
 & &
 \\ \hline

{$\mathrm{ N}_{8}^{\lambda,\mu \neq 0}$}  &
 $e_{1}  e_{1}=e_{1}+\frac{\lambda(\lambda-1)}{\mu}e_{3}$ &
 $e_{1} e_{2}=e_{2}$ &
 $e_{1} e_{3}=\lambda e_{3}$&
 & &

 \\
 &
 $e_{2} e_{1}=e_{2}$ & 
 $e_{3} e_{1}=\lambda e_{3}$&
 $e_{3}  e_{2}=e_{2}$& 
 & &
 \\
 &
  $e_{3}  e_{3}=\mu e_{3}$&
  &
  &
  & &
  \\ \hline

{$\mathrm{ N}_{9}^{\lambda}$}  &
 $e_{1}  e_{1}=e_{1}+\lambda e_{3}$ &
 $e_{1} e_{2}=e_{2}$ &
 $e_{2} e_{1}=e_{2}$&
 & &
 \\

 &
 $e_{3} e_{2}=e_{2}$ & 
 &
 &
 & &
 \\ \hline

{$\mathrm{ N}_{10}$}  &
 $e_{1}  e_{1}=e_{1}$ &
 $e_{1} e_{2}=e_{2}+e_{3}$ &
 $e_{2} e_{1}=e_{2}+e_{3}$&
 & &
 \\

 &
 $e_{3} e_{2}=e_{2}$ & 
 $e_{3} e_{3}=e_{3}$&
 &
 & &
 \\ \hline

{$\mathrm{ N}_{11}^{\lambda \neq 0}$}  &
 $e_{1}  e_{1}=e_{1}$ &
 $e_{2} e_{3}=\lambda e_{2}$ &
 $e_{3} e_{2}=(\lambda +1) e_{2}$&
 & &
 \\

 &
 $e_{3} e_{3}=\lambda e_{3}$ & 
 &
 &
 & &
 \\ \hline

$\mathrm{ N}_{12}$  &
 $e_{1}  e_{1}=e_{1}$ &
 $e_{2} e_{3}=-e_{2}$ &
 $e_{3} e_{3}=e_{2}-e_{3}$&
 & &
 \\ \hline

$\mathrm{ N}_{13}^{\mu}$  &
 $e_{1}  e_{1}=e_{1}+e_{2}$ &
 $e_{2} e_{3}=-e_{2}$ &
 $e_{3} e_{3}=\mu e_{2}-e_{3}$&
 & &
 \\ \hline

{$\mathrm{ N}_{14}$}  &
 $e_{1}  e_{1}=e_{1}$ &
 $e_{2} e_{2}=e_{3}$ &
 $e_{3} e_{2}=e_{2}$&
 & &
 \\

 &
 $e_{3} e_{3}=2e_{3}$ & 
 &
 &
 & &
 \\ \hline

{$\mathrm{ N}_{15}$}  &
 $e_{1}  e_{1}=e_{1}$ &
 $e_{1} e_{2}=e_{2}$ &
 $e_{1} e_{3}=e_{3}$&
 & &
 \\

 &
 $e_{2} e_{1}=e_{2}$ & 
 $e_{3} e_{1}=e_{3}$&
 $e_{3} e_{2}=e_{2}$&
 & &
 \\

 &
 $e_{3} e_{3}=e_{2}+e_{3}$ & 
 &
 &
 & &
 \\ \hline

$\mathrm{ N}_{16}$  &
 $e_{1}  e_{1}=e_{2}$ &
 $e_{2} e_{3}=-e_{2}$ &
 $e_{3} e_{3}=e_{2}-e_{3}$&
 & &
 \\ \hline

{$\mathrm{ N}_{17}^{\lambda }$}  &
 $e_{1}  e_{1}=e_{1}+e_{2}$ &
 $e_{1} e_{2}=e_{2}$ &
 $e_{1} e_{3}=e_{2}+e_{3}$&
 & &
\\

 &
 $e_{2} e_{1}=e_{2}$ & 
 $e_{3} e_{1}=e_{2}+e_{3}$&
 &
 & &
 \\

  &
 $e_{3} e_{2}=e_{2}$ & 
 $e_{3} e_{3}=-\lambda e_{2}+e_{3}$&
 &
 & &
 \\ \hline

 {$\mathrm{ N}_{18}$}  &
 $e_{1}  e_{3}=-e_{1}+e_{2}$ &
 $e_{2} e_{3}=-e_{2}$ &
 $e_{3} e_{1}=-e_{1}+e_{2}$&
 & &
 \\

 &
 $e_{3} e_{3}=e_{2}-e_{3}$ & 
 &
 &
 & &
 \\ \hline

 {$\mathrm{ N}_{19}^{\lambda}$}  &
 $e_{1}  e_{1}=e_{1}+{\lambda}e_{3}$ &
 $e_{1} e_{2}=e_{2}$ &
 $e_{1} e_{3}=e_{3}$&
 & &
 \\

 &
 $e_{2} e_{1}=e_{2}$ & 
 $e_{3} e_{1}=e_{3}$&
 $e_{3} e_{2}=e_{2}$&
 & &
 \\ \hline

{$\mathrm{ N}_{20}$}  &
 $e_{1}  e_{1}=e_{1}$ &
 $e_{1} e_{2}=e_{2}$ &
 $e_{1} e_{3}=e_{3}$&
 & &
 \\

 &
 $e_{2} e_{1}=e_{2}$ & 
 $e_{2} e_{2}=e_{3}$&
 $e_{3} e_{1}=e_{3}$&
 & &
 \\

 &
 $e_{3} e_{2}=e_{2}$ & 
 $e_{3} e_{3}=2e_{3}$&
 &
 & &
 \\ \hline

\multicolumn{7}{c}{\textbf{The left-symmetric algebras on $\rm{D}_{\ell}$:}} \\
\hline
\multicolumn{7}{c}{($\ell$ is as in (d) or $\ell=0$; for $\ell=0$ these are LSAs on $\mathrm{D}_0\cong\mathrm{N}$;} \\
\multicolumn{7}{c}{additional restrictions on $\ell$ and on the other parameters are indicated} \\
\multicolumn{7}{c}{in the first column; the algebras $\mathrm{D}_{22}$--$\mathrm{D}_{26}$ and $\mathrm{D}_{28}$--$\mathrm{D}_{31}$ exist only for the} \\
\multicolumn{7}{c}{indicated value of $\ell$; for $\ell=-1$ the parameter $\lambda$ of $\mathrm{D}_{1}^{-1,\lambda}$ and of $\mathrm{D}_{10}^{-1,\lambda}$ is taken up to sign)} \\
\hline
$\mathrm{ {D}_{1}^{\ell, \lambda }}$  &
 $e_{3}  e_{1}=e_{1}$ &
 $e_{3} e_{2}=\ell e_{2}$ &
 $e_{3} e_{3}=\lambda e_{3}$&
 & &
 \\ \hline
$\mathrm{ {D}_{2}^{\ell}}$  &
 $e_{3}  e_{1}=e_{1}$ &
 $e_{3} e_{2}=\ell e_{2}$ &
 $e_{3} e_{3}=e_{2}+\ell e_{3}$&
 & &
 \\ \hline
 {$\mathrm{ {D}_{3}^{\ell, \lambda \neq 0}}$}  &
 $e_{2}  e_{3}=\lambda e_{2}$ &
 $e_{3} e_{1}=e_{1}$&
 $e_{3} e_{2}=(\lambda +\ell)e_{2}$ &
 & &
 \\

 &
 $e_{3} e_{3}=\lambda e_{3}$&
 &
 &
 & &
 \\ \hline
$\mathrm{ {D}_{4}^{\ell\neq0}}$
&
 $e_{2}  e_{3}=-\ell e_{2}$ &
 $e_{3} e_{1}=e_{1}$ &
 $e_{3} e_{3}=e_{2} -\ell e_{3}$&
 & &
 \\ \hline
{$\mathrm{ {D}_{5}^{\ell}}$}  &
 $e_{2}  e_{2}=e_{3}$ &
 $e_{3} e_{1}=e_{1}$&
 $e_{3} e_{2}=\ell e_{2}$ &
 & &
 \\

 &
 $e_{3} e_{3}=2\ell e_{3}$&
 &
 &
 & &
 \\ \hline
{$\mathrm{ {D}_{6}^{\ell}}$}  &
 $e_{2}  e_{3}=e_{1}$ &
 $e_{3} e_{1}=e_{1}$&
 $e_{3} e_{2}=e_{1}+\ell e_{2}$ &
 & &
 \\

 &
 $e_{3} e_{3}=(1-\ell)e_{3}$&
 &
 &
 & &
 \\ \hline

 {$\mathrm{ {D}_{7}^{\ell}}$}  &
 $e_{2}  e_{3}=e_{2}$ &
 $e_{3} e_{1}=e_{1}$&
 $e_{3} e_{2}=(1+\ell)e_{2}$ &
 & &
 \\

 &
 $e_{3} e_{3}=e_{1}+e_{3}$&
 &
 &
 & &
 \\ \hline
 {$\mathrm{ {D}_{8}^{\ell}}$}  &
 $e_{1}  e_{1}=e_{2}$ &
 $e_{2} e_{3}=(2-\ell)e_{2}$&
 $e_{3} e_{1}=e_{1}$ &
 & &
 \\

 &
 $e_{3} e_{2}=2e_{2}$&
 $e_{3} e_{3}=(2-\ell)e_{3}$&
 &
 & &
 \\ \hline
{$\mathrm{ {D}_{9}^{\ell}}$}  &
 $e_{1}  e_{2}=e_{3}$ &
 $e_{2} e_{1}=e_{3}$&
 $e_{3} e_{1}=e_{1}$ &
 & &
 \\

 &
 $e_{3} e_{2}=\ell e_{2}$&
 $e_{3} e_{3}=(\ell+1)e_{3}$&
 &
 & &
 \\ \hline
{$\mathrm{ {D}}_{10}^{\ell,\lambda \neq 0}$}  &
 $e_{1}  e_{3}=\lambda e_{1}$ &
 $e_{2} e_{3}=\lambda e_{2}$&
 $e_{3} e_{1}=(\lambda +1)e_{1}$ &
 & &
 \\

 &
 $e_{3} e_{2}=(\lambda+\ell)e_{2}$&
 $e_{3} e_{3}=\lambda e_{3}$&
 &
 & &
 \\ \hline

{$\mathrm{ {D}_{11}^{\ell}}$}  &
 $e_{1}  e_{3}=-e_{1}$ &
 $e_{2} e_{3}=-e_{2}$&
 $e_{3} e_{2}=(\ell-1)e_{2}$ &
 & &
 \\

 &
 $e_{3} e_{3}=e_{1}-e_{3}$&
 &
 &
 & &
 \\ \hline
{$\mathrm{ {D}_{12}^{\ell\neq\frac{1}{2}}}$}  &
 $e_{2}  e_{2}=e_{1}$ &
 $e_{2} e_{3}=(1-2\ell)e_{2}$&
 $e_{3} e_{1}=e_{1}$ &
 & &
 \\

 &
 $e_{3} e_{2}=(1-\ell)e_{2}$ &
 $e_{3} e_{3}=(1-2\ell)e_{3}$&
 &
 & &
 \\ \hline

$\mathrm{ {D}_{13}^{\ell\neq\pm1}}$  &
 $e_{3}  e_{1}=e_{1}$ &
 $e_{3} e_{2}=\ell e_{2}$ &
 $e_{3} e_{3}= e_{1}+e_{3}$&
 & &
 \\ \hline
{$\mathrm{ {D}_{14}^{\ell\neq1}}$}  &
 $e_{1}  e_{3}=(\ell-1)e_{1}$ &
 $e_{2} e_{3}=e_{1}+(\ell-1)e_{2}$&
 $e_{3} e_{1}=\ell e_{1}$ &
 & &
 \\

 &
 $e_{3} e_{2}=e_{1}+(2\ell-1)e_{2}$&
 $e_{3} e_{3}=(\ell-1)e_{3}$&
 &
 & &
 \\ \hline

{$\mathrm{ {D}_{15}^{\ell\neq\pm1}}$}  &
 $e_{1}  e_{3}=(1-\ell)e_{1}+e_{2}$ &
 $e_{2} e_{3}=(1-\ell)e_{2}$&
 $e_{3} e_{1}=(2-\ell)e_{1}+e_{2}$ &
 & &
 \\

 &
 $e_{3} e_{2}=e_{2}$&
 $e_{3} e_{3}=(1-\ell)e_{3}$&
 &
 & &
 \\ \hline
{$\mathrm{ {D}_{16}^{\ell\neq \pm1,\lambda \neq 0}}$}  &
 $e_{1}  e_{3}=\lambda e_{1}$ &
 $e_{3} e_{1}=(\lambda+1)e_{1}$&
 $e_{3} e_{2}=\ell e_{2}$ &
 & &
 \\

 &
 $e_{3} e_{3}=\lambda e_{3}$&
 &
 &
 & &
 \\ \hline

$\mathrm{ {D}_{17}^{\ell\neq \pm1}}$  &
 $e_{1}  e_{3}=- e_{1}$ &
 $e_{3} e_{2}=\ell e_{2}$ &
 $e_{3} e_{3}=e_{1}-e_{3}$&
 & &
 \\ \hline

{$\mathrm{ {D}_{18}^{\ell\neq \pm1}}$}  &
 $e_{1}  e_{1}= e_{3}$ &
 $e_{3} e_{1}=e_{1}$&
 $e_{3} e_{2}=\ell e_{2}$ &
 & &
 \\

 &
 $e_{3} e_{3}=2e_{3}$&
 &
 &
 & &
 \\ \hline
{$\mathrm{ {D}_{19}^{\ell\neq \pm1}}$}  &
 $e_{1}  e_{3}= e_{2}$ &
 $e_{3} e_{1}=e_{1}+e_{2}$&
 $e_{3} e_{2}=\ell e_{2}$ &
 & &
 \\

 &
 $e_{3} e_{3}=(\ell-1)e_{3}$&
 &
 &
 & &
 \\ \hline
{$\mathrm{ {D}_{20}^{\ell\neq \pm1}}$}  &
 $e_{1}  e_{1}= e_{2}$ &
 $e_{1} e_{3}=(\ell-2)e_{1}$&
 $e_{3} e_{1}=(\ell-1)e_{1}$ &
 & &
 \\

 &
 $e_{3} e_{2}=\ell e_{2}$&
 $e_{3} e_{3}=(\ell-2)e_{3}$&
 &
 & &
 \\ \hline

{$\mathrm{ {D}_{21}^{\ell\neq\pm1,\frac{1}{2}}}$}  &
 $e_{1}  e_{3}=(2\ell-1)e_{1}$ &
 $e_{2} e_{2}=e_{1}$&
 $e_{3} e_{1}=2\ell e_{1}$ &
 & &
 \\

 &
 $e_{3} e_{2}=\ell e_{2}$ &
 $e_{3} e_{3}=(2\ell-1)e_{3}$&
 &
 & &
 \\ \hline

{$\mathrm{ {D}_{22}^{\lambda}}$, $\ell=\frac{1}{2}$}  &
 $e_{2}  e_{2}=e_{1}$ &
 $e_{3} e_{1}=e_{1}$&
 $e_{3} e_{2}=\frac{1}{2}e_{2}$ &
 & &
 \\

 &
 $e_{3} e_{3}=\lambda e_{3}$&
 &
 &
 & &
 \\ \hline
{$\mathrm{ {D}_{23}}$, $\ell=\frac{1}{2}$}  &
 $e_{2}  e_{2}=e_{1}$ &
 $e_{3} e_{1}=e_{1}$&
 $e_{3} e_{2}=\frac{1}{2}e_{2}$ &
 & &
 \\

 &
 $e_{3} e_{3}= e_{1}+e_{3}$&
 &
 &
 & &
 \\ \hline

$\mathrm{ {D}_{24}}$, $\ell=-1$  &
 $e_{2}  e_{3}=e_{2}$ &
 $e_{3} e_{1}=e_{1}$ &
 $e_{3} e_{3}=e_{1}+e_{2}+e_{3}$&
 & &
 \\ \hline
{$\mathrm{ {D}_{25}}$, $\ell=\frac{1}{2}$}  &
 $e_{2}  e_{3}=e_{1}$ &
 $e_{3} e_{1}=e_{1}$&
 $e_{3} e_{2}=e_{1}+\frac{1}{2}e_{2}$ &
 & &
 \\

 &
 $e_{3} e_{3}=e_{2}+\frac{1}{2}e_{3}$&
 &
 &
 & &
 \\ \hline

{$\mathrm{ {D}_{26}}$, $\ell=\frac{1}{3}$}  &
 $e_{2}  e_{2}=e_{3}$ &
 $e_{2} e_{3}=e_{1}$&
 $e_{3} e_{1}=e_{1}$ &
 & &
 \\

 &
 $e_{3} e_{2}=e_{1}+\frac{1}{3}e_{2}$ &
 $e_{3} e_{3}=\frac{2}{3}e_{3}$&
 &
 & &
 \\ \hline

{$\mathrm{ {D}_{27}^{\ell\neq0,\pm1}}$}  &
 $e_{1}  e_{3}=-\ell e_{1}$ &
 $e_{2} e_{3}=-\ell e_{2}$&
 $e_{3} e_{1}=(1-\ell)e_{1}$ &
 & &
 \\

 &
 $e_{3} e_{3}=e_{2}-\ell e_{3}$&
 &
 &
 & &
 \\ \hline

{$\mathrm{ {D}_{28}}$, $\ell=\frac{1}{2}$}  &
 $e_{1}  e_{3}=-\frac{1}{2}e_{1}$ &
 $e_{2} e_{3}=e_{1}-\frac{1}{2}e_{2}$&
 $e_{3} e_{1}=\frac{1}{2}e_{1}$ &
 & &
 \\

 &
 $e_{3} e_{2}=e_{1}$ &
 $e_{3} e_{3}=e_{2}-\frac{1}{2}e_{3}$&
 &
 & &
 \\ \hline
{$\mathrm{ {D}_{29}^{\lambda\neq0}}$, $\ell=\frac{1}{2}$}  &
 $e_{1}  e_{3}=\lambda e_{1}$ &
 $e_{2} e_{2}=e_{1}$&
 $e_{2} e_{3}=\lambda e_{2}$ &
 & &
 \\

 &
 $e_{3} e_{1}=(\lambda+1)e_{1}$ &
 $e_{3} e_{2}=(\lambda+\frac{1}{2})e_{2}$&
 $e_{3} e_{3}=\lambda e_{3}$&
 & &
 \\ \hline

{$\mathrm{ {D}_{30}}$, $\ell=\frac{1}{2}$}  &
 $e_{1}  e_{3}=-e_{1}$ &
 $e_{2} e_{2}=e_{1}$&
 $e_{2} e_{3}=-e_{2}$ &
 & &
 \\

 &
 $e_{3} e_{2}=-\frac{1}{2}e_{2}$&
 $e_{3} e_{3}=e_{1}-e_{3}$&
 &
 & &
 \\ \hline

{$\mathrm{ {D}_{31}}$, $\ell=\frac{1}{2}$}  &
 $e_{1}  e_{2}=e_{3}$ &
 $e_{2} e_{1}=e_{3}$&
 $e_{2} e_{2}=e_{1}$ &
 & &
 \\

 &
 $e_{3} e_{1}=e_{1}$&
 $e_{3} e_{2}=\frac{1}{2}e_{2}$&
 $e_{3} e_{3}=\frac{3}{2}e_{3}$&
 & &
 \\ \hline

{$\mathrm{ {D}_{32}^{\ell\neq 0,\pm1}}$}  &
 $e_{1}  e_{3}= \ell e_{1}$ &
 $e_{3} e_{1}=(1+\ell)e_{1}$&
 $e_{3} e_{2}=\ell e_{2}$ &
 & &
 \\

 &
 $e_{3} e_{3}=e_{2}+\ell e_{3}$&
 &
 &
 & &
 \\ \hline

\multicolumn{7}{c}{\textbf{The left-symmetric algebras on $\rm E$:}} \\
\hline

$\mathrm{ E_{1}^{\lambda }}$  &
 $e_{3}  e_{1}=e_{1}$ &
 $e_{3} e_{2}=e_{1}+e_{2}$ &
 $e_{3} e_{3}=\lambda e_{3}$&
 & &
 \\ \hline

$\mathrm{ {E}_{2}}$  &
 $e_{3}  e_{1}=e_{1}$ &
 $e_{3} e_{2}=e_{1}+e_{2}$ &
 $e_{3} e_{3}=e_{2}+e_{3}$&
 & &
 \\ \hline

$\mathrm{ {E}_{3}^{\lambda\neq0}}$  &
 $e_{2}  e_{3}=\lambda e_{1}$ &
 $e_{3} e_{1}=e_{1}$ &
 $e_{3} e_{2}=(\lambda+1)e_{1}+e_{2}$&
 & &
 \\ \hline

{$\mathrm{ {E}_{4}}$}  &
 $e_{2}  e_{2}=e_{3}$ &
 $e_{3} e_{1}=e_{1}$&
 $e_{3} e_{2}=e_{1}+e_{2}$ &
 & &
\\

&
 $e_{3} e_{3}=2e_{3}$ & 
 &
 &
 & &
 \\ \hline

{$\mathrm{ {E}_{5}}$}  &
 $e_{2}  e_{2}=e_{1}$ &
 $e_{2} e_{3}=-e_{1}-e_{2}$&
 $e_{3} e_{1}=e_{1}$ &
 & &
\\

&
 $e_{3} e_{3}=-e_{2}-e_{3}$ & 
 &
 &
 & &
 \\ \hline

{$\mathrm{ {E}_{6}^{\lambda\neq0}}$}  &
 $e_{1}  e_{3}=\lambda e_{1}$ &
 $e_{2} e_{3}=\lambda e_{2}$&
 $e_{3} e_{1}=(\lambda+1)e_{1}$ &
 & &
\\

&
 $e_{3} e_{2}=e_{1}+(\lambda+1)e_{2}$ & 
 $e_{3} e_{3}=\lambda e_{3}$&
 &
 & &
 \\ \hline

{$\mathrm{ {E}_{7}}$}  &
 $e_{1}  e_{3}=-e_{1}$ &
 $e_{2} e_{3}=-e_{2}$&
 $e_{3} e_{2}=e_{1}$ &
 & &
\\

&
 $e_{3} e_{3}=e_{2}-e_{3}$& 
 &
 &
 & &
 \\ \hline

{$\mathrm{{E}_{8}}$}  &
 $e_{1}  e_{3}=e_{1}$ &
 $e_{2} e_{2}=e_{1}$&
 $e_{3} e_{1}=2e_{1}$ &
 & &
\\

&
 $e_{3} e_{2}=e_{1}+e_{2}$ & 
 $e_{3} e_{3}=-e_{2}+e_{3}$ &
 &
 & &
 \\ \hline

\end{longtable}

\begin{remark}\label{remD}\leavevmode
{\rm 
\begin{enumerate}
\item Our numbering of the algebras differs from that of \cite{Bai}. The correspondence with the notation of \cite{Bai} for the left-symmetric algebras on $\mathrm{D}_\ell$ is as follows:
$\mathrm{D}_1,\ldots,\mathrm{D}_{11}$ are $(\mathrm{D}_l$-$1)$, $(\mathrm{D}_l$-$2)$, $(\mathrm{D}_l$-$4),\ldots,(\mathrm{D}_l$-$12)$;
$\mathrm{D}_{12}=(\mathrm{D}_{l\neq 1/2}$-$\mathrm{S}$-$1)$; $\mathrm{D}_{13}=(\mathrm{D}_l$-$3)$; $\mathrm{D}_{14}=(\mathrm{D}_l$-$14)$; $\mathrm{D}_{15}=(\mathrm{D}_l$-$15)$; 
$\mathrm{D}_{16},\ldots,\mathrm{D}_{19}$ are $(\mathrm{D}_{l\neq-1}$-$16),\ldots,(\mathrm{D}_{l\neq-1}$-$19)$; $\mathrm{D}_{20}=(\mathrm{D}_{l\neq-1}$-$21)$; $\mathrm{D}_{21}=(\mathrm{D}_{l\neq 1/2}$-$\mathrm{S}$-$2)$;
$\mathrm{D}_{22},\mathrm{D}_{23},\mathrm{D}_{25},\mathrm{D}_{28},\mathrm{D}_{29},\mathrm{D}_{30},\mathrm{D}_{31}$ are $(\mathrm{D}_{1/2}$-$\mathrm{S}$-$1),\ldots,(\mathrm{D}_{1/2}$-$\mathrm{S}$-$7)$; $\mathrm{D}_{24}=(\mathrm{D}_{-1}$-$\mathrm{T}$-$1)$; $\mathrm{D}_{26}=(\mathrm{D}_{1/3}$-$\mathrm{T}$-$1)$; $\mathrm{D}_{27}=(\mathrm{D}_l$-$13)$; $\mathrm{D}_{32}=(\mathrm{D}_{l\neq-1}$-$20)$.
Thus, apart from the algebras of the general families, there are additional algebras only for $\ell=\frac{1}{2}$ (namely $\mathrm{D}_{22}^{\lambda}$, $\mathrm{D}_{23}$, $\mathrm{D}_{25}$, $\mathrm{D}_{28}$, $\mathrm{D}_{29}^{\lambda}$, $\mathrm{D}_{30}$, $\mathrm{D}_{31}$), for $\ell=-1$ (namely $\mathrm{D}_{24}$) and for $\ell=\frac{1}{3}$ (namely $\mathrm{D}_{26}$); the families $\mathrm{D}_{12}^\ell$ and $\mathrm{D}_{21}^\ell$ are not defined for $\ell=\frac{1}{2}$ (for $\ell=\frac12$ both formulas give $\mathrm{D}_{22}^{0}$), and $\mathrm{D}_{16},\ldots,\mathrm{D}_{20}$, $\mathrm{D}_{32}$ are not defined for $\ell=-1$ in \cite{Bai}.

By \cite[Remark 2]{Bai}, the left-symmetric algebras on $\mathrm{D}_1$ are obtained by putting $\ell=1$ in the general families. For $\ell=1$ the isomorphisms of item~4 below (and the equality $\mathrm{D}_{14}^{1}=\mathrm{D}_{6}^{1}$) give
$\mathrm{D}_{13}^{1}\cong\mathrm{D}_{2}^{1}$, $\mathrm{D}_{14}^{1}=\mathrm{D}_{6}^{1}$, $\mathrm{D}_{15}^{1}\cong\mathrm{D}_{6}^{1}$, $\mathrm{D}_{16}^{1,\lambda}\cong\mathrm{D}_{3}^{1,\lambda}$, $\mathrm{D}_{17}^{1}\cong\mathrm{D}_{4}^{1}$, $\mathrm{D}_{18}^{1}\cong\mathrm{D}_{5}^{1}$, $\mathrm{D}_{19}^{1}\cong\mathrm{D}_{6}^{1}$, $\mathrm{D}_{20}^{1}\cong\mathrm{D}_{12}^{1}$, $\mathrm{D}_{21}^{1}\cong\mathrm{D}_{8}^{1}$, $\mathrm{D}_{27}^{1}\cong\mathrm{D}_{11}^{1}$, $\mathrm{D}_{32}^{1}\cong\mathrm{D}_{7}^{1}$; for this reason $\ell=1$ is excluded in these families. Similarly, for $\ell=-1$ one has $\mathrm{D}_{13}^{-1}\cong\mathrm{D}_{2}^{-1}$, $\mathrm{D}_{15}^{-1}\cong\mathrm{D}_{14}^{-1}$, $\mathrm{D}_{21}^{-1}\cong\mathrm{D}_{8}^{-1}$, $\mathrm{D}_{27}^{-1}\cong\mathrm{D}_{11}^{-1}$ (and the formulas of $\mathrm{D}_{16},\ldots,\mathrm{D}_{20}$, $\mathrm{D}_{32}$ for $\ell=-1$ give algebras isomorphic to $\mathrm{D}_{3}^{-1,-\lambda}$, $\mathrm{D}_{4}^{-1}$, $\mathrm{D}_{5}^{-1}$, $\mathrm{D}_{6}^{-1}$, $\mathrm{D}_{12}^{-1}$, $\mathrm{D}_{7}^{-1}$), so $\ell=-1$ is excluded in these families; moreover, the automorphism $e_1\mapsto e_2$, $e_2\mapsto e_1$, $e_3\mapsto -e_3$ of $\mathrm{D}_{-1}$ gives $\mathrm{D}_{1}^{-1,\lambda}\cong\mathrm{D}_{1}^{-1,-\lambda}$ and $\mathrm{D}_{10}^{-1,\lambda}\cong\mathrm{D}_{10}^{-1,-\lambda}$.

In \cite{Bai} the parameter $l$ of $\mathrm{D}_l$ is non-zero and the left-symmetric algebras on $\mathrm{N}$ are listed separately. By \cite[Remark 3]{Bai}, twenty of them are obtained by putting $\ell=0$ in the formulas of the families $\mathrm{D}_{1}$, $\mathrm{D}_{2}$, $\mathrm{D}_{3}$, $\mathrm{D}_{5},\ldots,\mathrm{D}_{21}$ (and writing them in the basis $e_2,\ e_1,\ e_3$, so that $[e_3,e_2]=e_2$): $\mathrm{D}_{1}^{0,\lambda}$, $\mathrm{D}_{2}^{0}$, $\mathrm{D}_{3}^{0,\lambda}$, $\mathrm{D}_{5}^{0},\ldots,\mathrm{D}_{21}^{0}$ are (N-1), (N-5), (N-9), (N-10), (N-11), (N-15), (N-36), (N-35), (N-38), (N-39), (N-16), (N-6), (N-40), (N-42), (N-18), (N-19), (N-20), (N-21), (N-44), (N-33) of \cite{Bai}, respectively. We do not list these algebras separately; instead, $\ell=0$ is allowed in the families on $\mathrm{D}_\ell$. These members with $\ell=0$ are pairwise non-isomorphic (and different values of $\lambda$ give non-isomorphic algebras in $\mathrm{D}_{1}^{0,\lambda}$, $\mathrm{D}_{3}^{0,\lambda}$, $\mathrm{D}_{10}^{0,\lambda}$, $\mathrm{D}_{16}^{0,\lambda}$), and they are not isomorphic to $\mathrm{N}_1,\ldots,\mathrm{N}_{20}$. Since $\mathrm{D}_{4}^{0}=\mathrm{D}_{27}^{0}=\mathrm{D}_{32}^{0}=\mathrm{D}_{2}^{0}$, the restriction $\ell\neq0$ is kept in $\mathrm{D}_{4}$, $\mathrm{D}_{27}$, $\mathrm{D}_{32}$.

\item In the algebra $\mathrm{D}_{21}^{\ell}$ the product $e_3e_3$ equals $(2\ell-1)e_3$. The algebra $\mathrm{D}_{25}$ is the algebra $(\mathrm{D}_{1/2}$-$\mathrm{S}$-$3)$ of \cite{Bai}, computed from the bijective $1$-cocycle given in \cite{Bai}; it has $e_2e_3=e_1$ and $e_2e_2=0$. Note that the multiplication table with $e_2e_2=e_1$ and $e_2e_3=0$ (and the other products as in $\mathrm{D}_{25}$) does not define a left-symmetric algebra: for it, $(e_2e_3)e_3-e_2(e_3e_3)=-e_1$, while $(e_3e_2)e_3-e_3(e_2e_3)=0$.

\item The algebra $\mathrm{N}_{10}$ is the algebra $\mathrm{N}_{10}^{1}$ of the family 
$\mathrm{N}_{10}^{\lambda}: e_1e_1=e_1, \ e_1e_2=e_2e_1=e_2+\lambda e_3, \ e_3e_2=e_2, \ e_3e_3=e_3$:
for $\lambda\neq0$, the basis $e_1, \lambda^{-1}e_2, e_3$ gives $\mathrm{N}_{10}^\lambda\cong\mathrm{N}_{10}^1$, and $\mathrm{N}_{10}^0=\mathrm{N}_{8}^{0,1}$.

\item The isomorphism $\mathrm{D}_\ell\cong\mathrm{D}_{1/\ell}$ induces isomorphisms between left-symmetric algebras on $\mathrm{D}_\ell$ and on $\mathrm{D}_{1/\ell}$. Namely, for $\ell\neq0$ (and $\lambda\neq0$ where needed) the following bases give the indicated isomorphisms:

\begin{longtable}{llll}
$\mathrm{D}_{13}^{\ell}\cong \mathrm{D}_{2}^{1/\ell}$: & $(e_2,\ \ell^{-2}e_1,\ \ell^{-1}e_3)$, &
$\mathrm{D}_{15}^{\ell}\cong \mathrm{D}_{14}^{1/\ell}$: & $(\ell^{-1}e_2,\ e_1,\ \ell^{-1}e_3)$, \\
$\mathrm{D}_{16}^{\ell,\lambda}\cong \mathrm{D}_{3}^{1/\ell,\lambda/\ell}$: & $(e_2,\ e_1,\ \ell^{-1}e_3)$, &
$\mathrm{D}_{17}^{\ell}\cong \mathrm{D}_{4}^{1/\ell}$: & $(e_2,\ \ell^{-2}e_1,\ \ell^{-1}e_3)$, \\
$\mathrm{D}_{18}^{\ell}\cong \mathrm{D}_{5}^{1/\ell}$: & $(e_2,\ \ell^{-1/2}e_1,\ \ell^{-1}e_3)$, &
$\mathrm{D}_{19}^{\ell}\cong \mathrm{D}_{6}^{1/\ell}$: & $(\ell^{-1}e_2,\ e_1,\ \ell^{-1}e_3)$, \\
$\mathrm{D}_{20}^{\ell}\cong \mathrm{D}_{12}^{1/\ell}$: & $(e_2,\ e_1,\ \ell^{-1}e_3)$, &
$\mathrm{D}_{21}^{\ell}\cong \mathrm{D}_{8}^{1/\ell}$: & $(e_2,\ e_1,\ \ell^{-1}e_3)$, \\
$\mathrm{D}_{27}^{\ell}\cong \mathrm{D}_{11}^{1/\ell}$: & $(\ell^{-2}e_2,\ e_1,\ \ell^{-1}e_3)$, &
$\mathrm{D}_{32}^{\ell}\cong \mathrm{D}_{7}^{1/\ell}$: & $(\ell^{-2}e_2,\ e_1,\ \ell^{-1}e_3)$. \\
\end{longtable}
\noindent If $\ell\neq0$ is as in (d) and $\ell\neq\pm1$, then $1/\ell$ is not as in (d) (for $0<|\ell|<1$ we have $|1/\ell|>1$, and for $\ell=e^{i\theta}$, $0<\theta<\pi$, we have $1/\ell=e^{-i\theta}$), so for these values of $\ell$ the above isomorphisms are not repetitions in the algebraic classification; for $\ell=\pm1$ they are repetitions, and these values are excluded in the corresponding families (see item~1). However, they show that the families $\mathrm{D}_{13}$, $\mathrm{D}_{15}$, $\mathrm{D}_{16}$, $\mathrm{D}_{17}$, $\mathrm{D}_{18}$, $\mathrm{D}_{19}$, $\mathrm{D}_{20}$, $\mathrm{D}_{21}$, $\mathrm{D}_{27}$, $\mathrm{D}_{32}$ have the same orbit closures as the families $\mathrm{D}_{2}$, $\mathrm{D}_{14}$, $\mathrm{D}_{3}$, $\mathrm{D}_{4}$, $\mathrm{D}_{5}$, $\mathrm{D}_{6}$, $\mathrm{D}_{12}$, $\mathrm{D}_{8}$, $\mathrm{D}_{11}$, $\mathrm{D}_{7}$, respectively (recall that the closure of a family does not change if the parameter $\ell$ runs over any Zariski dense subset of $\mathbb{C}\setminus\{0\}$).

\item For $\lambda\mu\neq0$, the basis $\lambda^{-1}e_1,\ e_2,\ -\mu\lambda^{-1}e_1+e_3$ gives $\mathrm{N}_{1}^{\lambda,\mu}\cong\mathrm{N}_{1}^{1/\lambda,\,-\mu/\lambda}$. Thus, for $\lambda\neq0$, the parameters $(\lambda,\mu)$ and $(1/\lambda,-\mu/\lambda)$ of $\mathrm{N}_{1}^{\lambda,\mu}$ give isomorphic algebras.

\item The algebras (N-28) and (N-29) of \cite{Bai} are $\mathrm{N}_{13}^{0}$ and $\mathrm{N}_{13}^{1}$. They come from the bijective $1$-cocycles of \cite[case (BI-7), $\lambda=-1$]{Bai} whose $(3,1)$-entry is $0$ or $1$; in fact, the corresponding matrix is a bijective $1$-cocycle for every value $\mu$ of this entry, and it gives the algebra $\mathrm{N}_{13}^{\mu}$. The parameter $\mu$ cannot be normalized to $0$ or $1$: the algebras $\mathrm{N}_{13}^{\mu}$, $\mu\in\mathbb{C}$, are pairwise non-isomorphic, and none of them is isomorphic to another algebra of the table. Hence the algebras $\mathrm{N}_{13}^{\mu}$ with $\mu\neq0,1$ are missing in \cite{Bai}.

\item Every restriction on the parameters which remains in the table is either needed for the multiplication table to make sense ($\mu\neq0$ in $\mathrm{N}_1$ and $\mathrm{N}_8$, $\lambda\neq1$ in $\mathrm{H}_9$), or excludes a member which is already a member of another family of the table:
$\mathrm{H}_6^0\cong\mathrm{H}_9^0$ (basis $e_2,\ e_1-e_2,\ -e_3$ of $\mathrm{H}_6^0$);
$\mathrm{N}_{11}^0=\mathrm{N}_2^0$;
$\mathrm{D}_3^{\ell,0}=\mathrm{D}_{10}^{\ell,0}=\mathrm{D}_{16}^{\ell,0}=\mathrm{D}_1^{\ell,0}$;
$\mathrm{D}_4^0=\mathrm{D}_{27}^0=\mathrm{D}_{32}^0=\mathrm{D}_2^0$;
$\mathrm{D}_{12}^{1/2}=\mathrm{D}_{21}^{1/2}=\mathrm{D}_{29}^0=\mathrm{D}_{22}^0$;
$\mathrm{D}_{14}^1=\mathrm{D}_6^1$;
the restrictions $\ell\neq\pm1$ of item~1;
$\mathrm{E}_3^0=\mathrm{E}_6^0=\mathrm{E}_1^0$.
The special members which are listed separately in \cite{Bai} although they belong to a broader family ((H-5), (N-3), (N-23), (N-25)$_0$, (E-3) and the algebras on $\mathrm{D}_1$) are the algebras $\mathrm{H}_9^0$, $\mathrm{N}_1^{0,\mu}$, $\mathrm{N}_8^{0,\mu}$, $\mathrm{N}_8^{0,1}$, $\mathrm{E}_3^{-1}$ and the members with $\ell=1$ of the families on $\mathrm{D}_\ell$, respectively.
\end{enumerate}
}
\end{remark}

\subsection{Geometric classification}

The following observation is standard (the map sending a left-symmetric algebra to its commutator is a $G$-equivariant morphism, hence it maps the closure of an orbit into the closure of the image orbit).

\begin{lemma}\label{thm4}
Let $A_1$ and $A_2$ be two $n$-dimensional left-symmetric algebras. Let $L_1$ and $L_2$
be the Lie algebras associated with $A_1$ and $A_2$, respectively.
If the algebra $A_1$ degenerates to the algebra $A_2$, then the Lie algebra $L_1$ degenerates to the Lie algebra $L_2$. 
\end{lemma}

We stress that Lemma~\ref{thm4} concerns degenerations between individual algebras (orbits) and cannot be applied directly to the closure of a family of algebras whose Lie algebras vary with the parameters. For instance, for $t\neq0$ the algebras $\mathrm{D}_1^{1+t,\lambda}$ are LSAs on $\mathrm{D}_{1+t}$, but the closure of the family $\mathrm{D}_1^{\ell,\lambda}$ contains the LSAs $\mathrm{E}_1^\lambda$ on $\mathrm{E}$, although no Lie algebra $\mathrm{D}_\ell$ degenerates to $\mathrm{E}$ (see the proof of Theorem~\ref{thm5}). For this reason, all the non-degenerations below are proved by Lemma~\ref{main2}, with explicitly given closed sets; the relations defining these sets may, of course, involve the commutators $c_{ij}^k-c_{ji}^k$.

\begin{theorem}\label{thm5}
    The variety of $3$-dimensional left-symmetric algebras has dimension  $10$, and it has  $19$  irreducible components, namely:
\begin{itemize}
\item two components of dimension $10$:
\begin{center}
        $\mathcal{C}_{1}=\overline{{\mathcal O}( \mathrm{ N}_{1}^{\lambda, \mu} )},$ \quad
        $\mathcal{C}_{2}=\overline{{\mathcal O}( \mathrm{ N}_{8}^{\lambda, \mu} )};$
\end{center}
\item seventeen components of dimension $9$:
\begin{center}
        $\mathcal{C}_{3}=\overline{{\mathcal O}( \mathrm{ T}_{1} )},$ \quad
        $\mathcal{C}_{4}=\overline{{\mathcal O}( \mathrm{ N}_{10} )},$ \quad
        $\mathcal{C}_{5}=\overline{{\mathcal O}( \mathrm{ N}_{14} )},$ \quad
        $\mathcal{C}_{6}=\overline{{\mathcal O}( \mathrm{ N}_{20} )},$ \quad
        $\mathcal{C}_{7}=\overline{{\mathcal O}( \mathrm{ D}_{26} )},$ \quad
        $\mathcal{C}_{8}=\overline{{\mathcal O}( \mathrm{ D}_{31} )},$ \quad
        $\mathcal{C}_{9}=\overline{{\mathcal O}( \mathrm{ D}_{1}^{\ell,\lambda} )},$ \quad
        $\mathcal{C}_{10}=\overline{{\mathcal O}( \mathrm{ D}_{3}^{\ell,\lambda} )},$ \quad
        $\mathcal{C}_{11}=\overline{{\mathcal O}( \mathrm{ D}_{10}^{\ell,\lambda} )},$ \quad
        $\mathcal{C}_{12}=\overline{{\mathcal O}( \mathrm{ D}_{5}^{\ell} )},$ \quad
        $\mathcal{C}_{13}=\overline{{\mathcal O}( \mathrm{ D}_{6}^{\ell} )},$ \quad
        $\mathcal{C}_{14}=\overline{{\mathcal O}( \mathrm{ D}_{8}^{\ell} )},$ \quad
        $\mathcal{C}_{15}=\overline{{\mathcal O}( \mathrm{ D}_{9}^{\ell} )},$ \quad
        $\mathcal{C}_{16}=\overline{{\mathcal O}( \mathrm{ D}_{12}^{\ell} )},$ \quad
        $\mathcal{C}_{17}=\overline{{\mathcal O}( \mathrm{ D}_{14}^{\ell} )},$ \quad
        $\mathcal{C}_{18}=\overline{{\mathcal O}( \mathrm{ D}_{22}^{\lambda} )},$ \quad
        $\mathcal{C}_{19}=\overline{{\mathcal O}( \mathrm{ D}_{29}^{\lambda} )}.$
\end{center}
\end{itemize}
In particular, there are exactly six rigid algebras: $\mathrm{T}_1=\mathbb{C}\oplus\mathbb{C}\oplus\mathbb{C}$, $\mathrm{N}_{10}$, $\mathrm{N}_{14}$, $\mathrm{N}_{20}$, $\mathrm{D}_{26}$ and $\mathrm{D}_{31}$.
Moreover, by Remark~\ref{remD}, $\mathcal{C}_{10}=\overline{{\mathcal O}( \mathrm{ D}_{16}^{\ell,\lambda} )}$, $\mathcal{C}_{12}=\overline{{\mathcal O}( \mathrm{ D}_{18}^{\ell} )}$, $\mathcal{C}_{13}=\overline{{\mathcal O}( \mathrm{ D}_{19}^{\ell} )}$, $\mathcal{C}_{14}=\overline{{\mathcal O}( \mathrm{ D}_{21}^{\ell} )}$, $\mathcal{C}_{16}=\overline{{\mathcal O}( \mathrm{ D}_{20}^{\ell} )}$ and $\mathcal{C}_{17}=\overline{{\mathcal O}( \mathrm{ D}_{15}^{\ell} )}$.
\end{theorem}

\begin{proof}
For every algebra (family) $J$ in the statement of the theorem, the following table gives the dimension of the closure $\overline{\mathcal{O}(J)}$ (see Section~1).

\begin{center}
\begin{tabular}{l|c||l|c||l|c}
\hline
$J$ & $\dim\overline{\mathcal{O}(J)}$ 
  & $J$ & $\dim\overline{\mathcal{O}(J)}$ 
  & $J$ & $\dim\overline{\mathcal{O}(J)}$ \\
\hline
$\mathrm{N}_{1}^{\lambda,\mu}$ & 10 
  & $\mathrm{D}_{26}$ & 9 
  & $\mathrm{D}_{6}^{\ell}$ & 9 \\
$\mathrm{N}_{8}^{\lambda,\mu}$ & 10 
  & $\mathrm{D}_{31}$ & 9 
  & $\mathrm{D}_{8}^{\ell}$ & 9 \\
$\mathrm{T}_{1}$ & 9 
  & $\mathrm{D}_{1}^{\ell,\lambda}$ & 9 
  & $\mathrm{D}_{9}^{\ell}$ & 9 \\
$\mathrm{N}_{10}$ & 9 
  & $\mathrm{D}_{3}^{\ell,\lambda}$ & 9 
  & $\mathrm{D}_{12}^{\ell}$ & 9 \\
$\mathrm{N}_{14}$ & 9 
  & $\mathrm{D}_{10}^{\ell,\lambda}$ & 9 
  & $\mathrm{D}_{14}^{\ell}$ & 9 \\
$\mathrm{N}_{20}$ & 9 
  & $\mathrm{D}_{5}^{\ell}$ & 9 
  & $\mathrm{D}_{22}^{\lambda}$ & 9 \\
& & & & $\mathrm{D}_{29}^{\lambda}$ & 9 \\
\hline
\end{tabular}
\end{center}

\noindent The sets $\mathcal{C}_1,\ldots,\mathcal{C}_{19}$ are irreducible and pairwise distinct. Indeed, if two of them coincided, then (since the union of the orbits of a family contains a dense open subset of its closure) a generic member of one family would be isomorphic to a member of the other family; this is not the case by the algebraic classification (for the families on $\mathrm{D}_\ell$ recall that the generic members of $\mathrm{D}_{22}^\lambda$ and $\mathrm{D}_{29}^\lambda$ are LSAs on $\mathrm{D}_{1/2}$, while the other families involve all values of $\ell$, and that no two of the families $\mathrm{D}_1,\mathrm{D}_3,\mathrm{D}_5,\mathrm{D}_6,\mathrm{D}_8,\mathrm{D}_9,\mathrm{D}_{10},\mathrm{D}_{12},\mathrm{D}_{14}$ are related by the isomorphisms of Remark~\ref{remD}). 

\medskip

Next we show that every $3$-dimensional left-symmetric algebra lies in $\mathcal{C}_1\cup\cdots\cup\mathcal{C}_{19}$, considering the left-symmetric algebras on each Lie algebra separately.

\begin{itemize}
\item {\it The abelian Lie algebra.} All the commutative associative algebras lie in $\mathcal{C}_3=\overline{\mathcal{O}(\mathrm{T}_1)}$, because of the following degenerations (and the obvious degeneration $\mathrm{T}_i\to\mathrm{T}_{12}$):

    \begin{longtable}{lcl|lcl}
$\mathrm{ T}_{1}$ & $\xrightarrow{\big(e_1,\ e_2,\ t e_3 \big)}$ & $\mathrm{ T}_{2}$ &
$\mathrm{ T}_{3}$ & $\xrightarrow{\big(e_1+e_2,\ te_1+ e_3,\ t^2 e_1 \big)}$ & $\mathrm{ T}_{6}$ \\
$\mathrm{ T}_{1}$ & $\xrightarrow{\big(e_1,\ e_2+e_3,\ t(e_2-e_3) \big)}$ & $\mathrm{ T}_{3}$ &
$\mathrm{ T}_{6}$ & $\xrightarrow{\big(e_1,\ te_2,\ e_3 \big)}$ & $\mathrm{ T}_{7}$ \\
$\mathrm{ T}_{1}$ & $\xrightarrow{\big(e_1,\ t(e_2-e_3),\ t^2(e_2+e_3) \big)}$ & $\mathrm{ T}_{4}$ &
$\mathrm{ T}_{2}$ & $\xrightarrow{\big(e_1,\ te_2,\ e_3 \big)}$ & $\mathrm{ T}_{8}$ \\
$\mathrm{ T}_{3}$ & $\xrightarrow{\big(e_2,\ e_3,\ t e_1 \big)}$ & $\mathrm{ T}_{5}$ &
$\mathrm{ T}_{4}$ & $\xrightarrow{\big(e_2,\ e_3,\ t e_1 \big)}$ & $\mathrm{ T}_{11}$ \\

\end{longtable}
\begin{longtable}{lcl}
$\mathrm{ T}_{1}$ 
  & $\xrightarrow{\big(t(e_1+2e_2+3e_3),\ t^2(e_1+4e_2+9e_3),\ t^3(e_1+8e_2+27e_3) \big)}$ 
  & $\mathrm{ T}_{9}$ \\
$\mathrm{ T}_{9}$ 
  & $\xrightarrow{\big(\frac{1}{\sqrt{2}}(te_1+e_2),\ -\frac{i}{\sqrt{2}}(te_1-e_2),\ te_3 \big)}$ 
  & $\mathrm{ T}_{10}$ \\
\end{longtable}

    \item {\it The Lie algebra $\mathrm{ H}$.} We have the following degenerations:

    \begin{longtable}{lcl|lcl}
$\mathrm{ H}_{2}$ & $\xrightarrow{\big(e_1,\ t e_2,\ t e_3 \big)}$ & $\mathrm{ H}_{1}$ &
$\mathrm{ H}_{2}$ & $\xrightarrow{\big(t^3e_1-\frac{t^3}{2}e_2+\frac{t^3}{4}e_3,\ t^4e_2,\ t^7e_3 \big)}$ & $\mathrm{ H}_{7}$ \\
$\mathrm{ H}_{3}$ & $\xrightarrow{\big(e_1,\ t e_2,\ t e_3 \big)}$ & $\mathrm{ H}_{4}$ &
$\mathrm{ H}_{2}$ & $\xrightarrow{\big(-t^2e_2, \ te_1-te_2,\ t^3e_3 \big)}$ & $\mathrm{ H}_{8}$ \\
$\mathrm{ H}_{2}$ & $\xrightarrow{\big(\frac{t^3}{1-t}e_2+\frac{t^3}{(t-1)^2}e_3,\ t e_1-\frac{t^2}{t-1} e_2,\ -\frac{t^4}{(t-1)^2} e_3 \big)}$ & $\mathrm{ H}_{5}$ &
$\mathrm{ H}_{2}$ & $\xrightarrow{\big(\frac{t^2}{1-\lambda }e_2,\ te_1 \ -\frac{\lambda  t}{\lambda -1}e_2-\frac{\lambda  t}{(\lambda -1)^2}e_3,\ \frac{t^3}{\lambda -1}e_3 \big)}$ & $\mathrm{ H}_{9}^\lambda $ \\
$\mathrm{ H}_{2}$ & $\xrightarrow{\big(te_1-\frac{t^2}{t-1}e_2,\ \frac{\lambda t}{1-t}e_2-\frac{\lambda(\lambda-t)}{(t-1)^2}e_3,\ \frac{\lambda t^2}{(t-1)^2}e_3\big)}$ & $\rm{H}_{6}^{\lambda}$ &
& & \\
\end{longtable}

\noindent    Moreover, the algebras $\mathrm{ H}_2$ and $\mathrm{ H}_3$ lie in $\mathcal{C}_2$:

\begin{longtable}{lcl}
$\mathrm{ {N}_{8}^{-\frac{1}{t^4+t-1}-t,  -t}}$ & $\xrightarrow{\big((-t^4-t+1)e_1+e_2+(t^4+t) e_3,\ \frac{1}{t^3} e_2+t e_3,\ -t^2 e_3 \big)}$ & $\mathrm{ {H}_{2}}$\\
$\mathrm{ {N}_{8}^{1, \frac{1}{t}}}$ & $\xrightarrow{\big( te_3,\ te_1+e_2,\ te_2 \big)}$ & $\mathrm{ {H}_{3}}$\\
\end{longtable}

    \item {\it The Lie algebra $\mathrm{ N}$.} We have the following degenerations:
\begin{longtable}{lcl|lcl}
$\mathrm{ N}_{1}^{t,\frac{t(t-1)}{\lambda}}$ & $\xrightarrow{\big(e_1,\ e_2,\ e_3\big)}$ & $\mathrm{ N}_{2}^\lambda$ &
$\mathrm{ N}_{8}^{\frac{t^3-1}{\lambda t}, \frac{1}{t}}$ & $\xrightarrow{\big(t e_3,\ e_2, \ \lambda e_1+e_3\big)}$ & $\mathrm{ N}_{11}^{\lambda}$ \\
$\mathrm{ N}_{1}^{{t^2},1+t}$ & $\xrightarrow{\big( e_1,\ -t e_2,\ e_2+e_3\big)}$ & $\mathrm{ N}_{3}$ &
$\mathrm{ N}_{8}^{\frac{t^5+t+1}{t^2}, \frac{1}{t^2}}$ & $\xrightarrow{\big(t^2 e_3,\ t e_2, \ -e_1+e_2+(t+1)e_3\big)}$ & $\mathrm{ N}_{12}$ \\
$\mathrm{ N}_{1}^{-\frac{1}{t},\ 1+t }$ & $\xrightarrow{\big(t^2 e_1,\ te_2,\ e_2+e_3\big)}$ & $\mathrm{ N}_{4}$ &
$\mathrm{ N}_{8}^{\frac{1}{t}-\mu,\frac{1}{t}}$ & $\xrightarrow{\big(te_2+te_3,\ -t e_2, \ -e_1+e_2+(1-\mu t)e_3\big)}$ & $\mathrm{ N}_{13}^{\mu}$ \\
$\mathrm{ N}_{1}^{1+t,\ \frac{t(1+t)}{\lambda}}$ & $\xrightarrow{\big(e_1,\ e_2,\ e_3\big)}$ & $\mathrm{ N}_{6}^{\lambda}$ & & & \\
$\mathrm{ N}_{8}^{t, \frac{t (t-1)}{\lambda}}$ & $\xrightarrow{\big(e_1,\ e_2, \ e_3\big)}$ & $\mathrm{ N}_{9}^{\lambda}$ &
$\mathrm{ N}_{8}^{1,1+t}$ & $\xrightarrow{\big(e_1+te_2+te_3,\ \frac{t}{1+t}e_3, \ e_2+e_3\big)}$ & $\mathrm{ N}_{15}$ \\
 & & & $\mathrm{ N}_{8}^{1+t, \frac{t(1+t)}{\lambda} }$ & $\xrightarrow{\big( e_1,\ e_2+te_3, \ e_3\big)}$ & $\mathrm{ N}_{19}^{\lambda}$ \\
\end{longtable}
    \begin{longtable}{lcl}
$\mathrm{ N}_{1}^{\frac{t (\lambda t+t+1)}{\lambda (t+1)},\ \frac{1}{t+1} }$ & $\xrightarrow{\big(e_1+\ \frac{t \left(\lambda^2 (t-1)+2 \lambda t+t\right)}{\lambda^2 (t-1)}e_2+te_3,\ \frac{t (t+1)}{\lambda} e_3,\ e_2+(1+t)e_3\big)}$ & $\mathrm{ N}_{5}^{\lambda}$\\
$\mathrm{ N}_{1}^{-\frac{2}{3t^2+1},\ 3}$ & $\xrightarrow{\big(-te_2-te_3,\ 2t^2e_3,\ (1+3t^2)e_1+e_2+e_3\big)}$ & $\mathrm{ N}_{7}$\\
$\mathrm{ N}_{8}^{\frac{t^4+2 t^3-t^2-2 t-1}{t^3-t-1}, \ 1+t}$ & $\xrightarrow{\big(\sqrt{t} e_2+\frac{t}{-t^3+t+1} e_3,\ \frac{t^3}{\left(-t^3+t+1\right)^2} e_3, \ -e_1+\frac{t+1}{\sqrt{t}}e_2+e_3\big)}$ & $\mathrm{ N}_{16}$\\
$\mathrm{ N}_{8}^{\frac{\lambda t+t-1}{\lambda t-1}, \frac{1}{1-\lambda t}}$ & $\xrightarrow{\big( e_1+te_2+te_3,\ t(\lambda t-1)e_3, \ e_2+(1-\lambda t)e_3\big)}$ & $\mathrm{ N}_{17}^{\lambda}$\\
$\mathrm{ N}_{8}^{\mu +\mu  t+1, \mu}$ & $\xrightarrow{\big( te_2+te_3,\ te_2, \ -e_1+e_2+(1+t)e_3\big)}$ & $\mathrm{ N}_{18}$\\
\end{longtable}

\noindent    Hence, the algebras $\mathrm{N}_1,\ldots,\mathrm{N}_{20}$ lie in $\mathcal{C}_1\cup\mathcal{C}_2\cup\mathcal{C}_4\cup\mathcal{C}_5\cup\mathcal{C}_6$ (recall that $\mathrm{N}_{10}^0=\mathrm{N}_8^{0,1}$, see Remark~\ref{remD}). The other LSAs on $\mathrm{N}$ are the members with $\ell=0$ of the families on $\mathrm{D}_\ell$; they are treated in the next item.

    \item {\it The Lie algebras $\mathrm{ D}_\ell$.} By Remark~\ref{remD}, the families $\mathrm{D}_{15}$, $\mathrm{D}_{16}$, $\mathrm{D}_{18}$, $\mathrm{D}_{19}$, $\mathrm{D}_{20}$, $\mathrm{D}_{21}$ lie in $\mathcal{C}_{17}$, $\mathcal{C}_{10}$, $\mathcal{C}_{12}$, $\mathcal{C}_{13}$, $\mathcal{C}_{16}$, $\mathcal{C}_{14}$, respectively. For the remaining algebras we have the following degenerations (since the structure constants of the families on $\mathrm{D}_\ell$ are polynomial in $\ell$, the members of a family with the excluded values of $\ell$, in particular its members with $\ell=0$ (which are LSAs on $\mathrm{N}$), lie in the closure of the family, see Section~1; so it suffices to consider generic values of $\ell$):
{    \setlength{\LTleft}{0pt}
\setlength{\LTright}{0pt}
\begin{longtable}{lcl|lcl}
$\mathrm{ {D}_{1}^{\ell,\ell + t}}$ 
  & $\xrightarrow{\big(e_1,\ te_2,\ -e_2+e_3 \big)}$ 
  & $\mathrm{ {D}_{2}^{\ell}}$ 
  & $\mathrm{ {D}_{1}^{\ell,1+t}}$ 
  & $\xrightarrow{\big(-t e_1,\ e_2,\ e_1+t e_2+e_3 \big)}$ 
  & $\mathrm{ {D}_{13}^{\ell}}$ \\[4pt]

$\mathrm{ {D}_{10}^{\frac{1+\ell t}{\ell },-\frac{1}{\ell}}}$ 
  & $\xrightarrow{\big(-\ell^2 t e_3,\ e_1,\ e_2+\ell e_3 \big)}$ 
  & $\mathrm{ {D}_{11}^{\ell}}$ 
  & $\mathrm{ {D}_{3}^{\ell,t-\ell}}$ 
  & $\xrightarrow{\big(e_1,\ t e_2,\ e_2+e_3\big)}$ 
  & $\mathrm{ {D}_{4}^{\ell}}$ \\[4pt]

$\mathrm{ {D}_{10}^{\frac{\ell}{t+1},-\ell}}$ 
  & $\xrightarrow{\big(e_1,\ \frac{\ell t}{t+1} e_3,\ -e_2+e_3 \big)}$ 
  & $\mathrm{ {D}_{27}^{\ell}}$ 
  & $\mathrm{ {D}_{3}^{\ell,t+1}}$ 
  & $\xrightarrow{\big(t e_1,\ e_2,\ -e_1+e_3 \big)}$ 
  & $\mathrm{ {D}_{7}^{\ell}}$ \\[4pt]

$\mathrm{ {D}_{16}^{\frac{\ell}{t+1},-\frac{t+2}{2 t+2}}}$ 
  & $\xrightarrow{\big(-\frac{1}{2} t (t+1) e_3,\ e_2,\ e_1+te_2+(t+1)e_3 \big)}$ 
  & $\mathrm{ {D}_{17}^{\ell}}$ 
  & $\mathrm{ {D}_{22}^{t+1}}$ 
  & $\xrightarrow{\big(t e_1,\ \sqrt{t} e_2,\ -e_1+e_3 \big)}$ 
  & $\mathrm{ {D}_{23}}$ \\[4pt]

$\mathrm{ {D}_{16}^{\ell+t, \ell}}$ 
  & $\xrightarrow{\big(e_1+t^2 e_2,\ -t e_3,\ e_2+e_3 \big)}$ 
  & $\mathrm{ {D}_{32}^{\ell}}$ 
  & $\mathrm{ {D}_{7}^{t-1}}$ 
  & $\xrightarrow{\big(e_1,\ t e_2,\ e_2+e_3 \big)}$ 
  & $\mathrm{ {D}_{24}}$ \\[4pt]

$\mathrm{ {D}_{29}^{-\frac{1}{2 (t+1)}}}$ 
  & $\xrightarrow{\big(\frac{t}{2} e_1,\ \frac{t}{2} e_2,\ -\frac{2}{2 t+1} e_1+e_2+(t+1)e_3 \big)}$ 
  & $\mathrm{ {D}_{28}}$ 
  & $\mathrm{ {D}_{29}^{-1-t^2}}$ 
  & $\xrightarrow{\big(t^2 e_3,\ t e_2,\ -e_1+e_3 \big)}$ 
  & $\mathrm{ {D}_{30}}$ \\[4pt]

$\mathrm{ {D}_{6}^{\frac{1}{2}+t}}$ 
  & $\xrightarrow{\big(e_1,\ e_2,\ -\frac{2}{t(1+2t)}e_1+\frac{1}{2t}e_2+e_3 \big)}$ 
  & $\mathrm{ {D}_{25}}$ 
  & & & \\[4pt]
\end{longtable}

}

\noindent Here, as in the table, the first superscript of an algebra on $\mathrm{D}_\ell$ is the value of $\ell$ and the second one is the value of $\lambda$; for instance, $\mathrm{D}_{10}^{\frac{1+\ell t}{\ell},-\frac{1}{\ell}}$ is the algebra $\mathrm{D}_{10}$ with $\ell$ replaced by $\frac{1+\ell t}{\ell}$ (which tends to $\frac{1}{\ell}$) and $\lambda=-\frac{1}{\ell}$, and $\mathrm{D}_{7}^{t-1}$ is the algebra $\mathrm{D}_7$ with $\ell=t-1$. Recall that $\mathrm{D}_{16}^{\ell,\lambda}$ lies in $\mathcal{C}_{10}$ by Remark~\ref{remD}, and $\mathrm{D}_7^\ell$ lies in $\mathcal{C}_{10}$ by the third row of the right column.

    \item {\it The Lie algebra $\mathrm{ E}$.} The following degenerations show that all LSAs on $\mathrm{E}$ lie in $\mathcal{C}_9\cup\mathcal{C}_{11}\cup\mathcal{C}_{12}\cup\mathcal{C}_{13}\cup\mathcal{C}_{14}\cup\mathcal{C}_{16}$ (recall that $\mathrm{D}_{21}^\ell$ lies in $\mathcal{C}_{14}$). Note that in these degenerations the source algebras are LSAs on $\mathrm{D}_{1+t}$; their structure constants in the basis $e_1,\ e_2-t^{-1}e_1,\ e_3$ involve the Lie bracket $[e_3,e_2-t^{-1}e_1]=e_1+(1+t)(e_2-t^{-1}e_1)$ of $\mathrm{D}_{1+t}$, which tends to the Lie bracket of $\mathrm{E}$.

    \begin{longtable}{lcl}
$\mathrm{ {D}_{1}^{1+t,\lambda}}$ & $\xrightarrow{\big(e_1,\ e_2-\frac{1}{t}e_1,\ e_3 \big)}$ & $\mathrm{ {E}_{1}^{\lambda}}$\\[4pt]
$\mathrm{ {E}_{1}^{t+1}}$ & $\xrightarrow{\big(te_1,\ te_2,\ -\frac{t^2+1}{t}e_1-e_2+e_3 \big)}$ & $\mathrm{ {E}_{2}}$\\[4pt]
$\mathrm{ {D}_{6}^{1+t}}$ & $\xrightarrow{\big(e_1,\ \lambda e_2-\frac{1}{t}e_1,\ e_3 \big)}$ & $\mathrm{ {E}_{3}^{\lambda}}$\\[4pt]
$\mathrm{ {D}_{5}^{1+t}}$ & $\xrightarrow{\big(e_1,\ e_2-\frac{1}{t}e_1,\ e_3 \big)}$ & $\mathrm{ {E}_{4}}$\\[4pt]
$\mathrm{ {D}_{12}^{1+t}}$ & $\xrightarrow{\big(e_1,\ e_2-\frac{1}{t}e_1,\ -\frac{1}{2t^2}e_1+\frac{t+1}{t}e_2+e_3 \big)}$ & $\mathrm{ {E}_{5}}$\\[4pt]
$\mathrm{ {D}_{10}^{1+t,\lambda}}$ & $\xrightarrow{\big(e_1,\ e_2-\frac{1}{t}e_1,\ e_3 \big)}$ & $\mathrm{ {E}_{6}^{\lambda}}$\\[4pt]
$\mathrm{ {E}_{6}^{t-1}}$ & $\xrightarrow{\big(\frac{t^2}{t^2-1}e_1,\ \frac{t^2}{t^2-1}e_2,\ e_1+\frac{t}{t^2-1}e_2+e_3 \big)}$ & $\mathrm{ {E}_{7}}$\\[4pt]
$\mathrm{ {D}_{21}^{1+t}}$ & $\xrightarrow{\big(e_1,\ e_2-\frac{1}{t}e_1,\ -\frac{1+2t}{2t^2}e_1+\frac{1+2t}{t}e_2+e_3 \big)}$ & $\mathrm{ {E}_{8}}$\\
\end{longtable}
\end{itemize}

It remains to prove the non-degenerations. By the dimensions computed above, $\mathcal{C}_i\not\subseteq\mathcal{C}_j$ whenever $i\neq j$ and $\dim\mathcal{C}_i\geq\dim\mathcal{C}_j$. Hence, by Section~1, it is sufficient to prove that none of the algebras (families) defining $\mathcal{C}_3,\ldots,\mathcal{C}_{19}$ belongs to $\mathcal{C}_1$ or $\mathcal{C}_2$, that is, that $\mathrm{N}_1^{\lambda,\mu}$ and $\mathrm{N}_8^{\lambda,\mu}$ do not degenerate to any of these algebras (families). This is done in the following table, where each non-degeneration is proved by Lemma~\ref{main2} with the set $\mathcal{W}$ given in the second column. Here $c_{ij}^k$ are the structure constants of the multiplication, the indices $i,j$ take all values in $\{1,2,3\}$, and each family is denoted by its name with the parameters of the classification; the non-degenerations whose target is a family are proved for generic values of its parameters, which is sufficient to show that $\mathcal{C}_i\not\subseteq\mathcal{C}_j$.

{\renewcommand{\arraystretch}{1.4}
\begin{longtable}{|rcl|l|}
\hline
\multicolumn{3}{|c|}{\textbf{Non-degeneration}} & \textbf{Relations defining} $\mathcal{W}$ \\
\hline
$\mathrm{N}_{1}^{\lambda,\mu},\ \mathrm{N}_{8}^{\lambda,\mu}$ 
  & $\not\to$ 
  & $\mathrm{T}_{1},\ \mathrm{N}_{14},\ \mathrm{N}_{20}$ 
  & $c_{33}^1=c_{33}^2=c_{33}^3=0$ \\
\hline
$\mathrm{N}_{1}^{\lambda,\mu},\ \mathrm{N}_{8}^{\lambda,\mu}$ 
  & $\not\to$ 
  & $\begin{array}{l}
       \mathrm{N}_{10},\ \mathrm{D}_{26},\ \mathrm{D}_{31},\ 
       \mathrm{D}_{1}^{\ell,\lambda},\ \mathrm{D}_{3}^{\ell,\lambda},\ 
       \mathrm{D}_{10}^{\ell,\lambda},\\
       \mathrm{D}_{5}^{\ell},\ \mathrm{D}_{6}^{\ell},\ 
       \mathrm{D}_{8}^{\ell},\ \mathrm{D}_{9}^{\ell},\ 
       \mathrm{D}_{12}^{\ell},\ \mathrm{D}_{14}^{\ell},\ 
       \mathrm{D}_{22}^{\lambda},\ \mathrm{D}_{29}^{\lambda}
     \end{array}$
  & $\begin{array}{l}
       c_{ij}^1=c_{ji}^1,\quad c_{ij}^2=c_{ji}^2,\\
       c_{i3}^1=c_{i3}^2=0
     \end{array}$ \\
\hline
\end{longtable}}

\noindent Both sets $\mathcal{W}$ are stable under the action of the invertible lower triangular matrices, since the subspace $\langle e_3\rangle$ is invariant under these matrices. More precisely, an algebra has a representation in the first set if and only if it contains a nonzero element $x$ with $xx=0$ (take $e_3=x$), and it has a representation in the second set if and only if it contains a one-dimensional left ideal $I$ ($AI\subseteq I$) containing all the commutators $xy-yx$ (take $\langle e_3\rangle=I$). 
\begin{itemize}
\item The algebras $\mathrm{N}_1^{\lambda,\mu}$ and $\mathrm{N}_8^{\lambda,\mu}$ have representations in both sets for all values of the parameters: in the basis $e_1,\ e_3,\ e_2$ they satisfy all the relations of the table (in both algebras $e_2e_2=0$, all commutators lie in $\langle e_2\rangle$, and $e_1e_2,\ e_2e_2,\ e_3e_2\in\langle e_2\rangle$).
\item The algebras $\mathrm{T}_1$, $\mathrm{N}_{14}$ and $\mathrm{N}_{20}$ have no representation in the first set: for $x=x_1e_1+x_2e_2+x_3e_3$ we have $xx=x_1^2e_1+x_2^2e_2+x_3^2e_3$ in $\mathrm{T}_1$, $xx=x_1^2e_1+x_2x_3e_2+(x_2^2+2x_3^2)e_3$ in $\mathrm{N}_{14}$, and $xx=x_1^2e_1+(2x_1x_2+x_2x_3)e_2+(2x_1x_3+x_2^2+2x_3^2)e_3$ in $\mathrm{N}_{20}$; in each case $xx=0$ implies $x=0$.
\item The algebras $\mathrm{D}_{26}$, $\mathrm{D}_{31}$, $\mathrm{D}_{1}^{\ell,\lambda}$, $\mathrm{D}_{3}^{\ell,\lambda}$, $\mathrm{D}_{10}^{\ell,\lambda}$, $\mathrm{D}_{5}^{\ell}$, $\mathrm{D}_{6}^{\ell}$, $\mathrm{D}_{8}^{\ell}$, $\mathrm{D}_{9}^{\ell}$, $\mathrm{D}_{12}^{\ell}$, $\mathrm{D}_{14}^{\ell}$, $\mathrm{D}_{22}^{\lambda}$, $\mathrm{D}_{29}^{\lambda}$ have no representation in the second set: for generic values of the parameters, their commutator algebra is the Lie algebra $\mathrm{D}_\ell$ (with $\ell\neq0$; $\ell=\frac{1}{3}$ for $\mathrm{D}_{26}$ and $\ell=\frac{1}{2}$ for $\mathrm{D}_{31}$, $\mathrm{D}_{22}^\lambda$, $\mathrm{D}_{29}^\lambda$), so the commutators span the two-dimensional space $\langle e_1,e_2\rangle$.
\item The algebra $\mathrm{N}_{10}$ has no representation in the second set: its commutators span $\langle e_2\rangle$, but $\langle e_2\rangle$ is not a left ideal, since $e_1e_2=e_2+e_3$.
\end{itemize}

\medskip

Thus every $3$-dimensional left-symmetric algebra belongs to $\mathcal{C}_1\cup\cdots\cup\mathcal{C}_{19}$, and the dimensions together with these non-degenerations give $\mathcal{C}_i\not\subseteq\mathcal{C}_j$ for $i\neq j$. Hence $\mathcal{C}_1,\ldots,\mathcal{C}_{19}$ are the irreducible components of the variety of $3$-dimensional left-symmetric algebras, and its dimension is $\max_i\dim\mathcal{C}_i=10$. Finally, an algebra is rigid if and only if its orbit closure is a component, and $\mathcal{C}_3,\ldots,\mathcal{C}_8$ are the only components which are orbit closures (the other components are closures of families of pairwise non-isomorphic algebras, so none of them is the closure of a single orbit).
\end{proof}

\section{Left-symmetric superalgebras}

\subsection{Algebraic classification}
\begin{definition}
    A superalgebra $S$ is called a left-symmetric superalgebra (LSSA) if the following identity holds
$$(x  y)  z - x  (y  z) = (-1)^{|x||y|}\Big((y  x)  z - y  (x  z)\Big)$$
for all homogeneous elements $x, y \in S_{\bar{0}}\cup S_{\bar{1}}$ and all $z \in S$.
\end{definition}

Let \( S = S_{\bar 0} \oplus S_{\bar 1} \) be an LSSA. Then its even part \( S_{\bar 0} \) forms a left-symmetric algebra; that is, the associator $(x,y,z)=(xy)z-x(yz)$ satisfies
\[
(x,y,z) = (y,x,z), \quad \text{for all } x,y,z \in S_{\bar 0}.
\]
It is well known that the super-commutator

$$[x,y] = x  y - (-1)^{|x||y|} y  x$$
(for homogeneous $x,y$, extended by bilinearity)
endows \( S \) with the structure of a Lie superalgebra, denoted by \( \mathfrak{G}(S) \). This Lie superalgebra is called the sub-adjacent Lie superalgebra of \( S \), and \( S \) is referred to as a compatible LSSA structure on \( \mathfrak{G}(S) \). Therefore, it is natural to study compatible LSSA structures on well-known Lie superalgebras. We recall the results of \cite{ZhangBai2012} that will be used below.

\begin{theorem}
Let $\mathfrak{g}=\mathfrak{g}_{\bar{0}} \oplus \mathfrak{g}_{\bar{1}}$ be a $2$-dimensional Lie superalgebra with a homogeneous basis $\{e, f\}$, where $e \in \mathfrak{g}_{\bar{0}}$ and $f \in \mathfrak{g}_{\bar{1}}$. Then $\mathfrak{g}$ is isomorphic to one 
of the mutually non-isomorphic Lie superalgebras given as follows:

\[
\begin{array}{ll}
(B_1): & [e, f] = f ;\\[4pt]
(B_2): & [f, f] = e ;\\[4pt]
(B_3): & \text{trivial Lie superalgebra.}
\end{array}
\]

\end{theorem}

\begin{theorem}
    Let $\mathfrak{g}=\mathfrak{g}_{\bar{0}} \oplus \mathfrak{g}_{\bar{1}}$ be a $3$-dimensional Lie superalgebra with $\operatorname{dim} \mathfrak{g}_{\bar{0}} = 1$. Assume that $\{e, f_1, f_2\}$ is a homogeneous basis of $\mathfrak{g}$ with $\mathfrak{g}_{\bar{0}}= \mathbb{C}e$ and $\mathfrak{g}_{\bar{1}}= \mathbb{C}f_1 \oplus \mathbb{C}f_2$. Then $\mathfrak{g}$ is isomorphic to one of the following Lie superalgebras:
 
        \[
        \begin{array}{ll}
        (C_1): & [e, f_2] = f_1 ;\\[4pt]
        (C_2)_\lambda: & [e, f_1] = f_1,\; [e, f_2] = \lambda f_2,\quad \lambda \in \mathbb{C} ;\\[4pt]
        (C_3): & [e, f_1] = f_1,\; [e, f_2] = f_1 + f_2 ;\\[4pt]
        (C_4): & \text{trivial Lie superalgebra;}\\[4pt]
        (C_5): & [f_1, f_2] = e ;\\[4pt]
        (C_6): & [f_1, f_1] = e .
        \end{array}
        \]
    
\end{theorem}

\noindent Note that $\mathrm{ (C_2)_\lambda} \cong \mathrm{ (C_2)_\tau}$ if and only if $\tau=\lambda$ or $\lambda \tau=1$ (an isomorphism $\mathrm{ (C_2)_\lambda} \to \mathrm{ (C_2)_{1/\lambda}}$, $\lambda\neq0$, is given by $e\mapsto \lambda^{-1} e$, $f_1\mapsto f_2$, $f_2\mapsto f_1$). Therefore, in the list above (and in Theorem~\ref{thm11} below), as in \cite{ZhangBai2012}, the parameter $\lambda$ should be considered up to the substitution $\lambda\mapsto\lambda^{-1}$.

\begin{theorem}
Let \( \mathfrak{g} = \mathfrak{g}_{\bar 0} \oplus \mathfrak{g}_{\bar 1} \) be a $3$-dimensional Lie superalgebra with \( \dim \mathfrak{g}_{\bar 0} = 2 \). Assume that \(\{e_1, e_2\}\) is a basis of \(\mathfrak{g}_{\bar 0}\) and \(\{f\}\) is a basis of \(\mathfrak{g}_{\bar 1}\). Then \( \mathfrak{g} \) is isomorphic to one of the mutually non-isomorphic Lie superalgebras given as follows:

\[
\begin{array}{ll}
(D_1): &  [e_1, f] = f ;\\[4pt]
(D_2): &  [f, f] = e_1 ;\\[4pt]
(D_3): & \text{trivial Lie superalgebra;}\\[4pt]
(D_4)_\mu: &  [e_1, e_2] = e_1,\; [e_2, f] = \mu f ;\\[4pt]
(D_5): &  [e_1, e_2] = e_1,\; [e_2, f] = -\frac{1}{2} f,\; [f, f] = e_1.
\end{array}
\]
\end{theorem}

All $2$- and $3$-dimensional LSSAs were classified in \cite{ZhangBai2012}. Since that classification includes certain superalgebras which appear as special cases of broader families, we omit such redundancies here and present only representatives of non-isomorphic families (up to the identifications of parameters stated below and in Remark~\ref{remS}); we also correct some inaccuracies of \cite{ZhangBai2012}, see Remark~\ref{remS}. In order to facilitate the comparison with \cite{ZhangBai2012}, we mostly keep the names of \cite{ZhangBai2012}; this is why some numbers are missing in the lists below. In the cases $(\mathrm{C}_4)$, $(\mathrm{C}_5)$ and $(\mathrm{A}_4)$ the numbering differs from \cite{ZhangBai2012}. All names which differ from those of \cite{ZhangBai2012}, and all superalgebras of \cite{ZhangBai2012} which are not listed separately here, are given in Table~\ref{tab:dict}. The superalgebra with zero multiplication (which is $\mathrm{B}_{3,1}$ in dimension $(1|1)$) is not listed in Theorems~\ref{thm11} and \ref{thm12}, which concern non-zero superalgebras; it belongs to the closure of every orbit.

\begin{table}[H]
\centering\small
\caption{Names used here and in \cite{ZhangBai2012} (only the names that differ, or the superalgebras of \cite{ZhangBai2012} which are not listed separately here).}\label{tab:dict}
\begin{tabular}{ll|ll|ll}
\hline
here & \cite{ZhangBai2012} & here & \cite{ZhangBai2012} & here & \cite{ZhangBai2012}\\ \hline
$\mathrm{B}_{1,3}^0$ & $\mathrm{B}_{1,1}$ & $\mathrm{C}_{4,5}$ & $\mathrm{C}_{4,6}$ & $\hat{\mathrm{A}}_{2,2}^0$ & $\hat{\mathrm{A}}_{2,1}$\\
$\mathrm{B}_{1,3}^1$ & $\mathrm{B}_{1,2}$ & zero & $\mathrm{C}_{4,2}$ & $\hat{\mathrm{A}}_{2,5}^1$ & $\hat{\mathrm{A}}_{2,4}$\\
$\mathrm{B}_{1,5}^{-1}$ & $\mathrm{B}_{1,4}$ (via $e\mapsto -e$) & $\mathrm{C}_{5,2}$ & $(\mathrm{C}_{5,1})_{k=1}$ & $\hat{\mathrm{A}}_{3,2}^0$ & $\hat{\mathrm{A}}_{3,1}$\\
$\mathrm{C}_{4,1}$ & $\mathrm{C}_{4,1}$ & $\mathrm{C}_{5,5}^{1/2}$ & $(\mathrm{C}_{5,1})_{k=0}$ & $\hat{\mathrm{A}}_{3,6}^1$ & $\hat{\mathrm{A}}_{3,5}$\\
$\mathrm{C}_{4,2}$ & $\mathrm{C}_{4,3}$ & $\mathrm{C}_{5,3}$ & $\mathrm{C}_{5,2}$ & zero & $\hat{\mathrm{A}}_{4,1}$\\
$\mathrm{C}_{4,3}$ & $\mathrm{C}_{4,4}$ & $\mathrm{C}_{5,4}$ & $\mathrm{C}_{5,3}$ & $\hat{\mathrm{A}}_{4,1}$ & $\hat{\mathrm{A}}_{4,2}$\\
$\mathrm{C}_{4,4}$ & $\mathrm{C}_{4,5}$ & $\mathrm{C}_{5,5}^k$ & $(\mathrm{C}_{5,4})_k$ & $\hat{\mathrm{A}}_{4,2}$ & $\hat{\mathrm{A}}_{4,3}$\\
$\hat{\mathrm{A}}_{1,3}^{0,0}$ & $\hat{\mathrm{A}}_{1,1}$ & $\hat{\mathrm{A}}_{1,3}^{k,0}$ & $(\hat{\mathrm{A}}_{1,2})_k$ & $\hat{\mathrm{A}}_{1,5}^{0,1}$ & $\hat{\mathrm{A}}_{1,4}$\\
$\hat{\mathrm{A}}_{7_{-1},3}^{k}$ & $(\hat{\mathrm{A}}_{6,1})_k$ & $\hat{\mathrm{A}}_{7_{-1},1}$ & $\hat{\mathrm{A}}_{6,2}$ & $\hat{\mathrm{A}}_{7_{-1},2}^{k}$ & $(\hat{\mathrm{A}}_{6,3})_k$\\
\hline
\end{tabular}
\end{table}

\begin{theorem}\label{thm10}
Let $S$ be a $2$-dimensional LSSA with one-dimensional even part. Then $S$ is isomorphic to exactly one of the following superalgebras.

\begin{itemize}
    \item If $S$ is a compatible LSSA structure on the Lie superalgebra of type $(\mathrm{B_1})$, then it is isomorphic to one of the following:
    \begin{longtable}{lclllll}
        $\mathrm{B}_{1,3}^{k}$&$:$&$ e e=ke $&$ e f=f $\\  
        $\mathrm{B}_{1,5}^{k \neq 0}$&$:$&$ e e=ke $&$ e f=(k+1)f $&$ f e=kf $\\
    \end{longtable}
    Here $k\in\mathbb{C}$ is arbitrary in $\mathrm{B}_{1,3}^{k}$ (for $k=0$ we get the superalgebra with the only nonzero product $ef=f$). The value $k=0$ is excluded in $\mathrm{B}_{1,5}^{k}$ only to avoid a repetition: $\mathrm{B}_{1,5}^{0}=\mathrm{B}_{1,3}^{0}$. The superalgebra $\mathrm{B}_{1,4}$ of \cite{ZhangBai2012} ($ee=e$, $fe=f$) is isomorphic to $\mathrm{B}_{1,5}^{-1}$ (via the basis $-e,\ f$), so it is not listed separately.
    
    \item If $S$ is a compatible LSSA structure on the Lie superalgebra of type $(\mathrm{B_2})$, then it is isomorphic to one of:
    \begin{longtable}{lclllll}
        $\mathrm{B}_{2,1}$&$:$&$ f f=\frac{1}{2}e $\\
        $\mathrm{B}_{2,2}$&$:$&$ e e=e $&$ e f=f $&$ f e=f $&$ f f=\frac{1}{2}e $ 
    \end{longtable}
    
    \item If $S$ is a compatible LSSA structure on the Lie superalgebra of type $(\mathrm{B_3})$, then it is isomorphic to one of:
    \begin{longtable}{lclllll}
        $\mathrm{B}_{3,1}$&$:$&$ \text{trivial} $\\
        $\mathrm{B}_{3,2}$&$:$&$ e e=e $\\
        $\mathrm{B}_{3,3}$&$:$&$ e e=e $&$ e f=f $&$ f e=f $ 
    \end{longtable}
\end{itemize}
\end{theorem}

\begin{theorem}\label{thm11}
    Let $S$ be a non-zero $3$-dimensional LSSA with one-dimensional even part. Then $S$ is isomorphic to one of the following superalgebras.
    \begin{itemize}
    \item If $S$ is a compatible LSSA structure on the Lie superalgebra of type $(\mathrm{C_1})$, then it is isomorphic to one of the following:
        \begin{longtable}{lclllll}
        $ \mathrm{C}_{1,1} $&$:$&$ e f_2=f_1 $&$ f_1 f_2=e $&$ f_2 f_1=-e $\\ 
        $ \mathrm{C}^k_{1,2} \ $&$:$&$ e f_2=(k+1)f_1 $&$ f_2 e=kf_1 $\\
        $ \mathrm{C}_{1,3}  $&$:$&$ e e=e $&$ e f_2=f_1 $\\ 
        $ \mathrm{C}_{1,4}  $&$:$&$ e e=e $&$ e f_1=f_1 $&$ e f_2=f_1+f_2 $&$ f_1 e=f_1 $&$ f_2 e=f_2 $ \\
        \end{longtable}
    
    \item If $S$ is a compatible LSSA structure on the Lie superalgebra of type $(\mathrm{C_2)_\lambda}$, then it is isomorphic to one of:
        \begin{longtable}{lclllll}
        $ \mathrm{C}_{{2}_\lambda,1}  $&$:$&$ e e=(1+\lambda)e $&$ e f_1=f_1 $&$ e f_2=\lambda f_2 $& &      \\
        &&$ f_1 f_2=e $&$ f_2 f_1=-e $&&&\\

        $ \mathrm{C}_{{2}_\lambda,2}  $&$:$&$ e e=(1-\lambda)e $&$ e f_1=f_1 $&$ e f_2=f_1+\lambda f_2 $& &\\ 
         & & $ f_2 e=f_1 $ &&&&\\

        $ \mathrm{C}_{{2}_\lambda,3 \quad (\lambda\neq1)}  $&$:$&$ e e=(\lambda-1)e $&$ e f_1=\lambda f_1 $&$ e f_2=f_1+(2\lambda-1) f_2 $&&$  $\\ 
        & & $ f_1 e=(\lambda-1)f_1 $ & $f_2 e=f_1+(\lambda-1)f_2$ &&&\\

        $ \mathrm{C}_{{2}_\lambda,4 \quad (\lambda\neq \pm 1)}  $&$:$&$ e e=(\lambda-1)e $&$ e f_1=f_1+f_2 $&$ e f_2=\lambda f_2 $& &\\
        && $ f_1 e=f_2 $ &&&&\\
    
        $ \mathrm{C}_{{2}_\lambda,5 \quad (\lambda\neq \pm 1)}  $&$:$&$ e e=(1-\lambda)e $&$ e f_1=(2-\lambda)f_1+f_2 $&$ e f_2=f_2 $& \\
        && $ f_1 e=(1-\lambda)f_1+f_2 $&$
        f_2 e=(1-\lambda)f_2 $ &&&\\ 

        $ \mathrm{C}^k_{{2}_\lambda,6}  $&$:$&$ e e=ke $&$ e f_1=f_1 $&$ e f_2=\lambda f_2 $\\

        $ \mathrm{C}^k_{{2}_\lambda,7 \quad (k\neq0, \lambda\neq \pm 1)}  $&$:$&$ e e=ke $&$ e f_1=(k+1)f_1 $&$ e f_2=\lambda f_2 $& \\
        && $ f_1 e=kf_1 $  &&&&\\

        $\mathrm{C}^k_{{2}_\lambda,8 \quad (k\neq 0)}  $&$:$&$ e e=ke $&$ e f_1=f_1 $&$ e f_2=(k+\lambda) f_2 $& \\
        && $ f_2 e=k f_2 $ &&&&\\

        $ \mathrm{C}^k_{{2}_\lambda,9 \quad (k\neq 0)}  $&$:$&$ e e=ke $&$ e f_1=(k+1)f_1 $&$ e f_2=(k+\lambda) f_2 $& \\
        && $ f_1 e=kf_1 $&$ f_2 e=kf_2 $ &&& \\
        \end{longtable}
    
    \item If $S$ is a compatible LSSA structure on the Lie superalgebra of type $(\mathrm{C_3})$, then it is isomorphic to one of:
        \begin{longtable}{lclllll}
        $ \mathrm{C}_{{3,1}}  $&$:$&$ e e=2e $&$ e f_1=f_1 $&$ e f_2=f_1+ f_2 $&$ f_1 f_2=e $&$ f_2 f_1=-e $\\
        $\mathrm{C}^k_{3,2} $&$:$&$ e e=ke $&$ e f_1=f_1 $&$ e f_2=f_1+ f_2 $\\
        $\mathrm{C}^k_{3,3 \quad (k\neq 0)} $&$:$&$ e f_1=f_1 $&$ e f_2=(k+1)f_1+f_2 $&$ f_2 e=kf_1 $\\
        $\mathrm{C}^k_{3,4 \quad (k\neq 0)}  $&$:$&$ e e=ke $&$ e f_1=(k+1)f_1 $&$ e f_2=f_1+(k+1)f_2 $&$ f_1 e=kf_1 $&$ f_2 e=kf_2 $\\
        \end{longtable}

    \item If $S$ is a compatible LSSA structure on the Lie superalgebra of type $(\mathrm{C_4})$, then it is isomorphic to one of:
        \begin{longtable}{lclllll}
        $ \mathrm{C}_{{4,1}}  $&$:$&$ f_1 f_2=e $&$ f_2 f_1=-e  $\\
        $ \mathrm{C}_{{4,2}}  $&$:$&$ e e=e $\\
        $ \mathrm{C}_{{4,3}} $&$:$&$ e e=e $&$ e f_2=f_2 $&$ f_2 e=f_2 $\\
        $ \mathrm{C}_{{4,4}}  $&$:$&$ e e=e $&$ e f_1=f_1 $&$ e f_2=f_2 $&$ f_1 e=f_1 $&$ f_2 e=f_2 $\\
        $ \mathrm{C}_{{4,5}}  $&$:$&$ e f_2=f_1 $&$ f_2 e=f_1 $\\
        \end{longtable}

    \item If $S$ is a compatible LSSA structure on the Lie superalgebra of type $(\mathrm{C_5})$, then it is isomorphic to one of:
        \begin{longtable}{lclllll}
        $ \mathrm{C}_{{5,2}}$&$:$&$ e f_1=f_2 $&$ f_1 e=f_2 $&$ f_2 f_1=e $\\
        $ \mathrm{C}_{{5,3}} $&$:$&$ e e=e $&$ e f_1=f_1 $&$ f_1 e=f_1 $&$ f_2 f_1=e $\\
        $ \mathrm{C}_{{5,4}}  $&$:$&$ e e=e $&$ e f_1=f_1+f_2 $&$ f_1 e=f_1+f_2 $&$ f_2 f_1=e $\\
        $ \mathrm{C}^k_{5,5}  $&$:$&$ f_1 f_2=(\frac{1}{2}-k)e $&$ f_2 f_1=(\frac{1}{2}+k)e $\\
        \end{longtable}
        Here $k\in\mathbb{C}$ and $\mathrm{C}^k_{5,5}\cong\mathrm{C}^{-k}_{5,5}$ (via the basis $e,\ f_2,\ f_1$). The superalgebra $\mathrm{C}^{1/2}_{5,5}$ ($f_2f_1=e$) is the superalgebra $(\mathrm{C}_{5,1})_{k=0}$ of \cite{ZhangBai2012}, which is therefore not listed separately.
        
    \item If $S$ is a compatible LSSA structure on the Lie superalgebra of type $(\mathrm{C_6})$, then it is isomorphic to one of:
        \begin{longtable}{lclllll}
        $\mathrm{C}_{{6,1}}  $&$:$&$ f_1 f_1=\frac{1}{2}e $&$ f_1 f_2=e $&$ f_2 f_1=-e $\\
        $\mathrm{C}_{{6,2}} $&$:$&$ f_1 f_1=\frac{1}{2}e $\\
        $\mathrm{C}_{{6,3}} $&$:$&$ e f_1=f_2 $&$ f_1 e=f_2 $&$ f_1 f_1=\frac{1}{2}e $\\
        $\mathrm{C}_{{6,4}}  $&$:$&$ e e=e $&$ e f_1=f_1 $&$ f_1 e=f_1 $&$ f_1 f_1=\frac{1}{2}e $\\
        \end{longtable}
    \end{itemize}
\end{theorem}

In Theorem~\ref{thm12} the superalgebras are grouped according to their even part $S_{\bar 0}=\langle e_1,e_2\rangle$, which is a $2$-dimensional LSA. We use the following notation for these LSAs (only non-zero products are given):
$(\mathrm{A}_1)$: $e_1e_1=e_1$, $e_2e_2=e_2$;
$(\mathrm{A}_2)$: $e_1e_1=e_1$, $e_1e_2=e_2e_1=e_2$;
$(\mathrm{A}_3)$: $e_1e_1=e_1$;
$(\mathrm{A}_4)$: the zero algebra;
$(\mathrm{A}_5)$: $e_1e_1=e_2$;
$(\mathrm{A}_7)_\lambda$, $\lambda\in\mathbb{C}$: $e_2e_1=-e_1$, $e_2e_2=\lambda e_2$;
$(\mathrm{A}_8)$: $e_1e_1=2e_1$, $e_1e_2=e_2$, $e_2e_2=e_1$;
$(\mathrm{A}_9)$: $e_1e_2=e_1$, $e_2e_2=e_2$;
$(\mathrm{A}_{10})_\lambda$, $\lambda\neq0$: $e_1e_2=\lambda e_1$, $e_2e_1=(\lambda-1)e_1$, $e_2e_2=e_1+\lambda e_2$;
$(\mathrm{A}_{11})$: $e_2e_1=-e_1$, $e_2e_2=e_1-e_2$.
In \cite{ZhangBai2012} the list also contains $(\mathrm{A}_6)$: $e_2e_1=-e_1$, $e_2e_2=-e_2$, and $(\mathrm{A}_7)_\lambda$ carries the restriction $\lambda\neq-1$. Since $(\mathrm{A}_6)=(\mathrm{A}_7)_{-1}$, we drop this restriction together with $(\mathrm{A}_6)$: the superalgebras with even part $(\mathrm{A}_6)$ are the members with $\lambda=-1$ of the families with even part $(\mathrm{A}_7)_\lambda$ (see Table~\ref{tab:dict}).

\begin{theorem}\label{thm12}
    Let $S$ be a non-zero $3$-dimensional LSSA with two-dimensional even part. Then $S$ is isomorphic to one of the following superalgebras (in all families, the parameters $k$ and $\lambda$ are arbitrary complex numbers unless a restriction is indicated).
    \begin{itemize}
    \item If $S$ is an LSSA with the even part ${\mathrm (A_1)}$, then it is isomorphic to one of:
        \begin{longtable}{lclllll}
        $\mathrm{\hat{A}}^{k_1,k_2}_{1,3}$&$:$&$ e_1 e_1=e_1 $&$ e_1 f=k_1f $&$ e_2 e_2=e_2 $&$ e_2 f=k_2f $\\
        $\mathrm{\hat{A}}^{k_1,k_2}_{1,5}$&$:$&$ e_1 e_1=e_1 $&$ e_1 f=k_1f $&$ e_2 e_2=e_2 $&$ e_2 f=k_2f $&$ f e_2=f $\\
        $\mathrm{\hat{A}}_{{1,6}}$&$:$&$ e_1 e_1=e_1 $&$ e_2 e_2=e_2 $&$ e_2 f=f $&$ f e_2=f $&$ f f=e_2 $\\
        \end{longtable}
        Here $k_1,k_2\in\mathbb{C}$ are arbitrary, and $\mathrm{\hat{A}}^{k_1,k_2}_{1,3}\cong\mathrm{\hat{A}}^{k_2,k_1}_{1,3}$ (via the basis $e_2,\ e_1,\ f$).

    \item If $S$ is an LSSA with the even part ${\mathrm (A_2)}$, then it is isomorphic to one of:
        \begin{longtable}{lcllllll}
        $\mathrm{\hat{A}}^k_{2,2}$&$:$&$ e_1 e_1=e_1 $&$ e_1 e_2=e_2 $&$ e_1 f=kf $&$ e_2 e_1=e_2 $\\
        $\mathrm{\hat{A}}^k_{2,3}$&$:$&$ e_1 e_1=e_1 $&$ e_1 e_2=e_2 $&$ e_1 f=kf $&$ e_2 e_1=e_2 $&$ e_2 f=f $\\
        $\mathrm{\hat{A}}^k_{2,5}$&$:$&$ e_1 e_1=e_1 $&$ e_1 e_2=e_2 $&$ e_1 f=kf $&$ e_2 e_1=e_2 $&$ f e_1=f $\\
        $\mathrm{\hat{A}}^k_{2,6}$&$:$&$ e_1 e_1=e_1 $&$ e_1 e_2=e_2 $&$ e_1 f=kf $&$ e_2 e_1=e_2 $&$ e_2 f=f $&$ f e_1=f $\\
        $\mathrm{\hat{A}}_{{2,7}}$&$:$&$ e_1 e_1=e_1 $&$ e_1 e_2=e_2 $&$ e_1 f=f $&$ e_2 e_1=e_2 $&$ f e_1=f $&$ f f=e_2 $\\
        \end{longtable}

    \item If $S$ is an LSSA with the even part ${\mathrm (A_3)}$, then it is isomorphic to one of:
        \begin{longtable}{lclllll}
        $\mathrm{\hat{A}}^k_{3,2}$&$:$&$ e_1 e_1=e_1 $&$ e_1 f=kf $\\
        $\mathrm{\hat{A}}^k_{3,3}$&$:$&$ e_1 e_1=e_1 $&$ e_1 f=kf $&$ e_2 f=f $\\
        $\mathrm{\hat{A}}_{{3,4}}$&$:$&$ e_1 e_1=e_1 $&$ f f=e_2 $\\
        $\mathrm{\hat{A}}^k_{3,6}$&$:$&$ e_1 e_1=e_1 $&$ e_1 f=kf $&$ f e_1=f $\\  
        $ \mathrm{\hat{A}}^k_{3,7}$&$:$&$ e_1 e_1=e_1 $&$ e_1 f=kf $&$ e_2 f=f $&$ f e_1=f $\\
        $\mathrm{\hat{A}}_{{3,8}}$&$:$&$ e_1 e_1=e_1 $&$ e_1 f=f  $&$ f e_1=f $&$ f f=e_1 $\\
        \end{longtable}

    \item If $S$ is an LSSA with the even part ${\mathrm (A_4)}$, then it is isomorphic to one of:
        \begin{longtable}{lclllll}
        $\mathrm{\hat{A}}_{{4,1}}$&$:$&$ e_1 f=f $\\
        $\mathrm{\hat{A}}_{{4,2}}$&$:$&$ f f=e_2 $\\
        \end{longtable}

    \item If $S$ is an LSSA with the even part ${\mathrm (A_5)}$, then it is isomorphic to one of:
        \begin{longtable}{lclllll}
        $\mathrm{\hat{A}}_{{5,1}}$&$:$&$  e_1 e_1=e_2 $&$ f f=e_2 $\\
        $\mathrm{\hat{A}}_{{5,2}}$&$:$&$  e_1 e_1=e_2 $\\
        $\mathrm{\hat{A}}_{{5,3}}$&$:$&$  e_1 e_1=e_2 $&$ e_1 f=f $\\
        $\mathrm{\hat{A}}_{{5,4}}$&$:$&$  e_1 e_1=e_2 $&$ e_2 f=f $\\
        \end{longtable}
        
    \item If $S$ is an LSSA with the even part ${\mathrm (A_7)}$, then it is isomorphic to one of:
        \begin{longtable}{lclllll}
        $\mathrm{\hat{A}}_{{{7}_\lambda,1}}$&$:$&$  e_2 e_1=-e_1 $&$ e_2 e_2=\lambda e_2 $&$ e_2 f=-\frac{1}{2}f $&$ f f=e_1 $\\
        $\mathrm{\hat{A}}_{{{7}_\lambda,2} \quad (\lambda\neq0)}^k$&$:$&$  e_2 e_1=-e_1 $&$ e_2 e_2=\lambda e_2 $&$ e_2 f=kf $&$ f e_2=\lambda f $\\
        $\mathrm{\hat{A}}_{{{7}_\lambda,3} }^k$&$:$&$  e_2 e_1=-e_1 $&$ e_2 e_2=\lambda e_2 $&$ e_2 f=kf $\\
        \end{longtable}

    \item If $S$ is an LSSA with the even part ${\mathrm (A_8)}$, then it is isomorphic to one of:
        \begin{longtable}{lclllll}
        $\mathrm{\hat{A}}_{{{8},1}}^k$&$:$&$  e_1 e_1=2e_1 $&$ e_1 e_2=e_2 $&$ e_1 f=kf $&$ e_2 e_2=e_1 $\\
        \end{longtable}

    \item If $S$ is an LSSA with the even part ${\mathrm (A_9)}$, then it is isomorphic to one of:
        \begin{longtable}{lclllll}
        $\mathrm{\hat{A}}_{{{9},1}}^k$&$:$&$  e_1 e_2=e_1 $&$ e_2 e_2=e_2 $&$ e_2 f=kf $\\
        $\mathrm{\hat{A}}_{{{9},2}}^k$&$:$&$  e_1 e_2=e_1 $&$ e_2 e_2=e_2 $&$ e_2 f=kf $&$f e_2=f $\\
        $\mathrm{\hat{A}}_{{{9},3}}$&$:$&$  e_1 e_2=e_1 $&$ e_2 e_2=e_2 $&$ e_2 f=\frac{1}{2}f $&$ f e_2=f $&$ f f=e_1 $\\
        \end{longtable}

    \item If $S$ is an LSSA with the even part ${\mathrm (A_{10})_\lambda}$, then it is isomorphic to one of:
        \begin{longtable}{lclllll}
        $\mathrm{\hat{A}}_{{{10}_\lambda,1}}^k$&$:$&$  e_1 e_2=\lambda e_1 $&$ e_2 e_1=(\lambda-1)e_1 $&$ e_2 e_2=e_1+\lambda e_2 $&$ e_2 f=kf $\\

        $\mathrm{\hat{A}}_{{{10}_\lambda,2}}^k$&$:$&$  e_1 e_2=\lambda e_1 $&$ e_2 e_1=(\lambda-1)e_1 $&$ e_2 e_2=e_1+\lambda e_2 $&$ e_2 f=kf $\\ && $ f e_2=\lambda f $\\
        
        $\mathrm{\hat{A}}_{{{10}_\lambda,3}}$&$:$&$  e_1 e_2=\lambda e_1 $&$ e_2 e_1=(\lambda-1)e_1 $&$ e_2 e_2=e_1+\lambda e_2 $&$ e_2 f=(\lambda-\frac{1}{2})f $\\
        &&$ f e_2=\lambda f$& $ f f=e_1 $\\
        \end{longtable}
    \item If $S$ is an LSSA with the even part ${\mathrm (A_{11})}$, then it is isomorphic to one of:
        \begin{longtable}{lclllll}
        $\mathrm{\hat{A}}_{{{11},1}}^k$&$:$&$  e_2 e_1=-e_1 $&$ e_2 e_2=e_1-e_2 $&$ e_2 f=kf $\\
        $\mathrm{\hat{A}}_{{{11},2}}$&$:$&$  e_2 e_1=-e_1 $&$ e_2 e_2=e_1-e_2 $&$ e_2 f=-\frac{1}{2}f $&$ f f=e_1 $\\
        $\mathrm{\hat{A}}_{{{11},3}}^k$&$:$&$  e_2 e_1=-e_1 $&$ e_2 e_2=e_1-e_2 $&$ e_2 f=kf $&$ f e_2=-f $\\
        \end{longtable}
    \end{itemize}
\end{theorem}

\begin{remark}\label{remS}
{\rm
We comment on the differences between Theorems~\ref{thm10}--\ref{thm12} and the classification of \cite{ZhangBai2012}.
\begin{enumerate}
\item {\it A missing family.} In \cite[Table 11]{ZhangBai2012} the family $\mathrm{\hat{A}}_{1,3}^{k_1,k_2}$ carries the restrictions $k_1,k_2\neq0$, $k_1\leq k_2$, $k_1\neq-k_2$. The condition $k_1\neq-k_2$ is not correct: the superalgebras $\mathrm{\hat{A}}_{1,3}^{k,-k}$ ($k\neq0$) are left-symmetric and are not isomorphic to any other superalgebra of the list (the unordered pair $\{k_1,k_2\}$ of the eigenvalues on $S_{\bar 1}$ of the left multiplications by the two primitive idempotents $e_1,e_2$ of the even part is an isomorphism invariant, and the conditions $S_{\bar1}S_{\bar0}=0$, $S_{\bar1}S_{\bar1}=0$ distinguish $\mathrm{\hat{A}}_{1,3}$ from $\mathrm{\hat{A}}_{1,5}$ and $\mathrm{\hat{A}}_{1,6}$); so they are missing in \cite{ZhangBai2012}. The condition $k_1\leq k_2$ is meaningless over $\mathbb{C}$; it is replaced by the isomorphism $\mathrm{\hat{A}}^{k_1,k_2}_{1,3}\cong\mathrm{\hat{A}}^{k_2,k_1}_{1,3}$. Similarly, the condition $k\geq0$ in $(\mathrm{C}_{5,4})_k$ of \cite{ZhangBai2012} (our $\mathrm{C}^k_{5,5}$) is replaced by the isomorphism $\mathrm{C}^k_{5,5}\cong\mathrm{C}^{-k}_{5,5}$.
\item {\it Restrictions excluding separately listed superalgebras.} In several places \cite{ZhangBai2012} imposes a restriction on the parameter of a family and lists the excluded member as a separate superalgebra. We drop such restrictions and do not list the special members separately:
$\mathrm{B}_{1,1}=\mathrm{B}^0_{1,3}$, $\mathrm{B}_{1,2}=\mathrm{B}^1_{1,3}$;
$\mathrm{B}_{1,4}\cong\mathrm{B}^{-1}_{1,5}$;
$(\mathrm{C}_{5,1})_{k=0}=\mathrm{C}^{1/2}_{5,5}$;
$\mathrm{\hat{A}}_{1,1}=\mathrm{\hat{A}}^{0,0}_{1,3}$, $(\mathrm{\hat{A}}_{1,2})_k=\mathrm{\hat{A}}^{k,0}_{1,3}$, $\mathrm{\hat{A}}_{1,4}=\mathrm{\hat{A}}^{0,1}_{1,5}$;
$\mathrm{\hat{A}}_{2,1}=\mathrm{\hat{A}}^0_{2,2}$, $\mathrm{\hat{A}}_{2,4}=\mathrm{\hat{A}}^1_{2,5}$, $\mathrm{\hat{A}}_{3,1}=\mathrm{\hat{A}}^0_{3,2}$, $\mathrm{\hat{A}}_{3,5}=\mathrm{\hat{A}}^1_{3,6}$;
$(\mathrm{A}_6)=(\mathrm{A}_7)_{-1}$, hence $(\mathrm{\hat{A}}_{6,1})_k=\mathrm{\hat{A}}^k_{7_{-1},3}$, $\mathrm{\hat{A}}_{6,2}=\mathrm{\hat{A}}_{7_{-1},1}$, $(\mathrm{\hat{A}}_{6,3})_k=\mathrm{\hat{A}}^k_{7_{-1},2}$.
The remaining restrictions separate two different families with a common member, or are needed for the multiplication table to make sense: $\mathrm{B}^0_{1,5}=\mathrm{B}^0_{1,3}$; $\mathrm{C}^0_{3,3}=\mathrm{C}^0_{3,4}=\mathrm{C}^0_{3,2}$; $\mathrm{C}^0_{2_\lambda,i}=\mathrm{C}^0_{2_\lambda,6}$ ($i=7,8,9$); $\mathrm{C}_{2_1,3}=\mathrm{C}_{2_1,2}$; $\mathrm{C}_{2_{\pm1},4}\cong\mathrm{C}_{2_{\pm1},2}$, $\mathrm{C}_{2_{\pm1},5}\cong\mathrm{C}_{2_{\pm1},3}$, $\mathrm{C}^k_{2_{\pm1},7}\cong\mathrm{C}^{\pm k}_{2_{\pm1},8}$; $\mathrm{\hat{A}}^k_{7_0,2}=\mathrm{\hat{A}}^k_{7_0,3}$; $\lambda\neq0$ in $(\mathrm{A}_{10})_\lambda$, since $(\mathrm{A}_{10})_0\cong(\mathrm{A}_7)_0$.
\item {\it The family $(\mathrm{C}_2)_\lambda$.} By \cite[Proposition 4.3]{ZhangBai2012}, $\mathrm{C}_{2_\lambda,i}\cong\mathrm{C}_{2_\tau,i}$ if and only if $\lambda=\tau$ or $\lambda\tau=1$. This is not accurate: the isomorphism $(\mathrm{C}_2)_\lambda\cong(\mathrm{C}_2)_{1/\lambda}$ ($\lambda\neq0$) maps $\mathrm{C}_{2_\lambda,2}$, $\mathrm{C}_{2_\lambda,3}$, $\mathrm{C}^k_{2_\lambda,7}$ to $\mathrm{C}_{2_{1/\lambda},4}$, $\mathrm{C}_{2_{1/\lambda},5}$, $\mathrm{C}^{k/\lambda}_{2_{1/\lambda},8}$ (see the beginning of the proof of Theorem~\ref{thm14}), and it gives $\mathrm{C}_{2_\lambda,1}\cong\mathrm{C}_{2_{1/\lambda},1}$ (basis $\lambda^{-1}e,\ f_2,\ -\lambda^{-1}f_1$), $\mathrm{C}^k_{2_\lambda,6}\cong\mathrm{C}^{k/\lambda}_{2_{1/\lambda},6}$ and $\mathrm{C}^k_{2_\lambda,9}\cong\mathrm{C}^{k/\lambda}_{2_{1/\lambda},9}$ (basis $\lambda^{-1}e,\ f_2,\ f_1$). Thus the list of Theorem~\ref{thm11} is a list up to isomorphism if $\lambda$ is taken modulo $\lambda\mapsto\lambda^{-1}$, apart from the following repetitions for $\lambda=-1$: $\mathrm{C}^k_{2_{-1},6}\cong\mathrm{C}^{-k}_{2_{-1},6}$ and $\mathrm{C}^k_{2_{-1},9}\cong\mathrm{C}^{-k}_{2_{-1},9}$.
\item {\it The superalgebras on $(\mathrm{C}_3)$.} The superalgebra $\mathrm{C}^k_{3,4}$ ($k\neq0$) is a compatible structure on $(\mathrm{C}_3)$; it is not isomorphic to $\mathrm{C}_{3,1}$, $\mathrm{C}^k_{3,2}$, $\mathrm{C}^k_{3,3}$, since its right multiplication $R_e\colon x\mapsto xe$ on $S_{\bar 1}$ is $k\,\mathrm{id}\neq0$, and $\mathrm{C}^0_{3,4}=\mathrm{C}^0_{3,2}$. A direct computation of all compatible LSSA structures on $(\mathrm{C}_3)$ shows that they are exactly the superalgebras $\mathrm{C}_{3,1}$, $\mathrm{C}^k_{3,2}$, $\mathrm{C}^k_{3,3}$, $\mathrm{C}^k_{3,4}$ of Theorem~\ref{thm11}.
\end{enumerate}
With these changes, Theorems~\ref{thm10}--\ref{thm12} give all superalgebras of dimensions $(1|1)$, $(1|2)$ and $(2|1)$ up to isomorphism, the only repetitions being the stated isomorphisms within families.
}
\end{remark}

\subsection{Geometric classification}

In this subsection we use the notation of Section~1; in particular, the dimension of an orbit is computed with respect to the group $G=\mathrm{GL}(V_{\bar 0})\times\mathrm{GL}(V_{\bar 1})$, that is, $\mathrm{GL}_1\times\mathrm{GL}_1$, $\mathrm{GL}_1\times\mathrm{GL}_2$ and $\mathrm{GL}_2\times\mathrm{GL}_1$ for the dimensions $(1|1)$, $(1|2)$ and $(2|1)$, respectively. Non-degenerations are proved by Lemma~\ref{main2}, and the structure constants $c_{ij}^k$ refer to the basis $x_1=e,\ x_2=f$ in dimension $(1|1)$, $x_1=e,\ x_2=f_1,\ x_3=f_2$ in dimension $(1|2)$, and $x_1=e_1,\ x_2=e_2,\ x_3=f$ in dimension $(2|1)$.

\begin{theorem}\label{thm13}
    The variety of $2$-dimensional left-symmetric superalgebras with one-dimensional even part has dimension  $2$ and it has  $3$  irreducible components:
\begin{center}
        $\mathcal{C}_{1}=\overline{{\mathcal O}( {\mathrm {B}_{1,3}^k} )},$ \quad
        $\mathcal{C}_{2}=\overline{{\mathcal O}( {\mathrm {B}_{1,5}^k} )},$ \quad
        $\mathcal{C}_{3}=\overline{{\mathcal O}( {\mathrm {B}_{2,2}} )}.$
\end{center}
All of them have dimension $2$. In particular, there is only one rigid superalgebra, $\mathrm{B}_{2,2}$.
\end{theorem}

\begin{proof}
A $(1|1)$-dimensional superalgebra is given by $ee=a e$, $ef=b f$, $fe=c f$, $ff=d e$, where $(a,b,c,d)\in\mathbb{C}^4$. A direct computation shows that the left-symmetric superidentity is equivalent to the system
$$c(a-c)=0, \quad d(a-c)=0, \quad d(b-c)=0, \quad d(a-2b+c)=0,$$
whose set of solutions is the union of the three planes
$$\Pi_1=\{c=d=0\},\qquad \Pi_2=\{a=c,\ d=0\},\qquad \Pi_3=\{a=b=c\}.$$
Since none of these planes is contained in the union of the other two, they are the irreducible components of the variety, and its dimension is $2$. The group $\mathrm{GL}_1\times\mathrm{GL}_1$ acts by $(a,b,c,d)\mapsto(sa,sb,sc,s^{-1}u^2d)$ (for the basis $se,\ uf$). Hence $\Pi_1=\overline{\mathcal{O}(\mathrm{B}^k_{1,3})}$ (the superalgebras $ee=a e$, $ef=b f$ with $b\neq0$ are isomorphic to $\mathrm{B}_{1,3}^{a/b}$), $\Pi_2=\overline{\mathcal{O}(\mathrm{B}^k_{1,5})}$ (the superalgebras $ee=a e$, $ef=b f$, $fe=a f$ with $b\neq a$ are isomorphic to $\mathrm{B}_{1,5}^{a/(b-a)}$ if $a\neq0$, and to $\mathrm{B}_{1,3}^{0}=\mathrm{B}_{1,5}^{0}$ if $a=0$), and $\Pi_3=\overline{\mathcal{O}(\mathrm{B}_{2,2})}$, since the points of $\Pi_3$ with $ad\neq0$ form the orbit of $\mathrm{B}_{2,2}$. In particular, $\dim\mathcal{O}(\mathrm{B}_{1,3}^k)=\dim\mathcal{O}(\mathrm{B}_{1,5}^k)=1$, $\dim\mathcal{O}(\mathrm{B}_{2,2})=2$, and $\mathrm{B}_{2,2}$ is the only rigid superalgebra. Note that, since the three irreducible closed sets $\mathcal{C}_1,\mathcal{C}_2,\mathcal{C}_3$ are pairwise distinct and have the same dimension, none of them is contained in another one (see Section~1); hence no non-degeneration has to be proved in this case. For completeness, we also give degenerations to the remaining superalgebras (the trivial superalgebra $\mathrm{B}_{3,1}$ lies in every orbit closure):
\begin{longtable}{lcl|lcl}
${\mathrm {B}_{1,3}^{\frac{1}{t}}}$ & $\xrightarrow{\big(te,\ f \big)}$ & ${\mathrm {B}_{3,2}}$ &
${\mathrm {B}_{1,5}^{\frac{1}{t}}}$ & $\xrightarrow{\big(te,\ f \big)}$ & ${\mathrm {B}_{3,3}}$ \\
\multicolumn{6}{c}{${\mathrm {B}_{2,2}} \xrightarrow{\big(t^2e,\ tf \big)} {\mathrm {B}_{2,1}}$} \\
\end{longtable}
\end{proof}

It is straightforward to see that if $S_1\to S_2$, then \( \mathfrak{G}(S_1)\to\mathfrak{G}(S_2)  \). As in the non-graded case, this fact concerns individual orbits only, and it is not needed in the proofs below.

\begin{theorem}\label{thm14}
    The variety of $3$-dimensional left-symmetric superalgebras with one-dimensional even part has dimension  $5$, and it has  $8$  irreducible components:
\begin{center}
        $\mathcal{C}_{1}=\overline{{\mathcal O}( {\mathrm C}_{{2}_\lambda,1} )},$ \quad
        $\mathcal{C}_{2}=\overline{{\mathcal O}( {\mathrm C}_{{2}_\lambda,2} )}=\overline{{\mathcal O}( {\mathrm C}_{{2}_\lambda,4} )},$ \quad
        $\mathcal{C}_{3}=\overline{{\mathcal O}( {\mathrm C}_{{2}_\lambda,3} )}=\overline{{\mathcal O}( {\mathrm C}_{{2}_\lambda,5} )},$ \\[4pt]
        $\mathcal{C}_{4}=\overline{{\mathcal O}( {\mathrm C}_{{2}_\lambda,6}^k )},$ \quad
        $\mathcal{C}_{5}=\overline{{\mathcal O}( {\mathrm C}_{{2}_\lambda,7}^k )}=\overline{{\mathcal O}( {\mathrm C}_{{2}_\lambda,8}^k )},$ \quad
        $\mathcal{C}_{6}=\overline{{\mathcal O}( {\mathrm C}_{{2}_\lambda,9}^k )},$ \\[4pt]
        $\mathcal{C}_{7}=\overline{{\mathcal O}( {\mathrm C}_{5,4})},$ \quad
        $\mathcal{C}_{8}=\overline{{\mathcal O}( {\mathrm C}_{5,5}^k)}.$
\end{center}
The components $\mathcal{C}_1,\ldots,\mathcal{C}_7$ have dimension $5$, and $\dim\mathcal{C}_8=4$. In particular, there is only one rigid superalgebra, $\mathrm{C}_{5,4}$.
\end{theorem}

\begin{proof}
For $\lambda\neq0$, the isomorphism $(\mathrm{C}_2)_\lambda\cong(\mathrm{C}_2)_{1/\lambda}$ induces the following isomorphisms of LSSAs (we indicate the basis giving the isomorphism):
\begin{longtable}{ll}
$\mathrm{C}_{2_\lambda,2}\cong\mathrm{C}_{2_{1/\lambda},4}$: & $(\lambda^{-1}e,\ f_2,\ \lambda^{-1}f_1)$, \\
$\mathrm{C}_{2_\lambda,3}\cong\mathrm{C}_{2_{1/\lambda},5}$: & $(\lambda^{-1}e,\ f_2,\ \lambda^{-1}f_1)$, \\
$\mathrm{C}^k_{2_\lambda,7}\cong\mathrm{C}^{k/\lambda}_{2_{1/\lambda},8}$: & $(\lambda^{-1}e,\ f_2,\ f_1)$. \\
\end{longtable}
Hence the families $\mathrm{C}_{2_\lambda,2}$ and $\mathrm{C}_{2_\lambda,4}$ (respectively, $\mathrm{C}_{2_\lambda,3}$ and $\mathrm{C}_{2_\lambda,5}$; $\mathrm{C}^k_{2_\lambda,7}$ and $\mathrm{C}^k_{2_\lambda,8}$) have the same orbit closure.

\medskip 

The dimensions of the closures of the families are given in the following table (see Section~1):
\begin{longtable}{l|c||l|c}
\hline
$J$ & $\dim\overline{\mathcal{O}(J)}$ & $J$ & $\dim\overline{\mathcal{O}(J)}$ \\
\hline
${\mathrm C}_{{2}_\lambda,1}$ & 5 & ${\mathrm C}^k_{{2}_\lambda,7}$ & 5 \\
${\mathrm C}_{{2}_\lambda,2}$ & 5 & ${\mathrm C}^k_{{2}_\lambda,9}$ & 5 \\
${\mathrm C}_{{2}_\lambda,3}$ & 5 & ${\mathrm C}_{5,4}$ & 5 \\
${\mathrm C}^k_{{2}_\lambda,6}$ & 5 & ${\mathrm C}^k_{5,5}$ & 4 \\
\hline
\end{longtable}
\noindent The irreducible closed sets $\mathcal{C}_1,\ldots,\mathcal{C}_8$ are pairwise distinct, since generic members of different families are not isomorphic (they are distinguished, for example, by the Lie superalgebra $\mathfrak{G}(S)$, by the right multiplications $R_e\colon f\mapsto fe$ on $S_{\bar1}$, which are zero, nilpotent, non-semisimple with one eigenvalue, of rank one, and scalar for $\mathcal{C}_1$ and $\mathcal{C}_4$, $\mathcal{C}_2$, $\mathcal{C}_3$, $\mathcal{C}_5$, $\mathcal{C}_6$, respectively, and by the product $S_{\bar 1}S_{\bar 1}$, which is nonzero only for $\mathcal{C}_1$, $\mathcal{C}_7$, $\mathcal{C}_8$).

\medskip 

The following degenerations show that all the other superalgebras of Theorem~\ref{thm11} belong to $\mathcal{C}_1\cup\cdots\cup\mathcal{C}_8$. In particular, $\mathrm{C}^k_{3,4}$ lies in $\mathcal{C}_6$ by the first degeneration of the table.

\begin{longtable}{lcl|lcl}
${\mathrm C}_{2_{1+t},9}^{k}$ & $\xrightarrow{\big(e, \ t f_2, \ f_1+f_2 \big)}$ & ${\mathrm C}_{3,4}^{k}$ &
${\mathrm C}_{2_{1+t},6}^{k}$ & $\xrightarrow{\big(e, \ tf_2, \ f_1+f_2 \big)}$ & ${\mathrm C}_{3,2}^k$ \\[4pt]

${\mathrm C}_{5,4} $ & $\xrightarrow{\big(e, \ t f_1+ t f_2, \ \frac{1}{t} f_2 \big)}$ & ${\mathrm C}_{5,3}$ &
${\mathrm C}_{2_{1+t},2}$ & $\xrightarrow{\big(e, \ f_1, \ kf_2-\frac{1}{t}f_1 \big)}$ & ${\mathrm C}_{3,3}^k$ \\[4pt]

${\mathrm C}_{6,4} $ & $\xrightarrow{\big(t^2 e, \ t f_1, \ f_2 \big)}$ & ${\mathrm C}_{6,2}$ &
${\mathrm C}_{3,1}$ & $\xrightarrow{\big(te, \ tf_1+(t^2-t)f_2, \ f_1+tf_2 \big)}$ & ${\mathrm C}_{4,1}$ \\[4pt]

${\mathrm C}_{6,4} $ & $\xrightarrow{\big(t^2 e, \ t f_1+f_2, \ -t^2 f_2 \big)}$ & ${\mathrm C}_{6,3}$ &
${\mathrm C}_{1,3}$ & $\xrightarrow{\big(e, \ \frac{t-1}{t^2}f_1+f_2, \ f_1+tf_2 \big)}$ & ${\mathrm C}_{4,2}$ \\[4pt]

${\mathrm C}_{4,3} $ & $\xrightarrow{\big(t e, \ f_1, \ -\frac{1}{t} f_1+ f_2\big)}$ & ${\mathrm C}_{4,5}$ &
${\mathrm C}_{2_{1},8}^{\frac{1}{t}}$ & $\xrightarrow{\big(te, \ f_1+tf_2, \ f_2 \big)}$ & ${\mathrm C}_{4,3}$ \\[4pt]

${\mathrm C}_{3,1} $ & $\xrightarrow{\big(te, \ t f_1+ t^2f_2, \ f_2 \big)}$ & ${\mathrm C}_{1,1}$ &
${\mathrm C}_{2_{-1},9}^{\frac{1}{t}}$ & $\xrightarrow{\big(te, \ f_1+f_2, \ f_1-f_2 \big)}$ & ${\mathrm C}_{4,4}$ \\[4pt]

${\mathrm C}_{3,3}^k $ & $\xrightarrow{\big(te, \ t f_1+ tf_2, \ f_2 \big)}$ & ${\mathrm C}_{1,2}^k$ &
${\mathrm C}_{5,4}$ & $\xrightarrow{\big((1+t^2)e, \ \frac{1}{\sqrt{2}}f_1+\frac{1+t^2}{\sqrt{2}}f_2, \ \frac{t}{\sqrt{2}}f_2 \big)}$ & ${\mathrm C}_{6,4}$ \\[4pt]

${\mathrm C}_{3,2}^{\frac{1}{t}} $ & $\xrightarrow{\big(te, \ t f_1+ t^2f_2, \ f_2 \big)}$ & ${\mathrm C}_{1,3}$ &
${\mathrm C}_{5,5}^{-\frac{1}{2t^2}}$ & $\xrightarrow{\big(e, \ tf_1+\frac{1}{{2t}}f_2, \ -2t^3f_1+tf_2 \big)}$ & ${\mathrm C}_{6,1}$ \\[4pt]

${\mathrm C}_{2_{-1},9}^{\frac{1}{t}}$ & $\xrightarrow{\big(te, \ tf_1+f_2, \ f_1-\frac{1}{t}f_2 \big)}$ & ${\mathrm C}_{1,4}$ &
${\mathrm C}_{2_{1+t},1}$ & $\xrightarrow{\big(\frac{2}{2+t}e, \ \frac{t}{2+t}f_1+ \frac{1}{2+t}f_2, \ -f_1+\frac{1}{t}f_2 \big)}$ & ${\mathrm C}_{3,1}$ \\[4pt]
\multicolumn{6}{c}{${\mathrm C}_{5,4} \xrightarrow{\big(t e, \ \sqrt{t^2+1} f_1+ \frac{t^2}{\sqrt{t^2+1}} f_2, \ \frac{t}{\sqrt{t^2+1}} f_2 \big)} {\mathrm C}_{5,2}$} \\
\end{longtable}

\noindent Here ${\mathrm C}_{2_{\mu},i}^{k}$ denotes the superalgebra ${\mathrm C}_{2_\lambda,i}^{k}$ with $\lambda=\mu$. 

\medskip

By the table of dimensions, $\mathcal{C}_i\not\subseteq\mathcal{C}_j$ whenever $i\neq j$ and $\dim\mathcal{C}_i\geq\dim\mathcal{C}_j$. Hence it is sufficient to prove that the family ${\mathrm C}^k_{5,5}$ does not belong to $\mathcal{C}_1,\ldots,\mathcal{C}_7$. This is done in the following table, where each non-degeneration is proved by Lemma~\ref{main2} with the set $\mathcal{W}$ given in the second column ($c_{ij}^k$ are the structure constants in the basis $x_1=e,\ x_2=f_1,\ x_3=f_2$, and the non-degenerations are proved for generic values of $k$):

{\renewcommand{\arraystretch}{1.4}
\begin{longtable}{|rcl|l|}
\hline
\multicolumn{3}{|c|}{\textbf{Non-degeneration}} & \textbf{Relations defining} $\mathcal{W}$ \\
\hline
$\begin{array}{l}
   {\mathrm C}_{{2}_\lambda,2},\ {\mathrm C}_{{2}_\lambda,3},\ 
   {\mathrm C}^k_{{2}_\lambda,6},\\ {\mathrm C}^k_{{2}_\lambda,7},\
   {\mathrm C}^k_{{2}_\lambda,9},\ {\mathrm C}_{5,4}
 \end{array}$
  & $\not\to$ 
  & ${\mathrm C}^k_{5,5}$ 
  & $c_{22}^1=c_{23}^1=0$ \\
\hline
${\mathrm C}_{{2}_\lambda,1}$ 
  & $\not\to$ 
  & ${\mathrm C}^k_{5,5}$ 
  & $c_{22}^1=c_{33}^1=0,\quad c_{23}^1+c_{32}^1=0$ \\
\hline
\end{longtable}

}

Consequently, $\mathcal{C}_1,\ldots,\mathcal{C}_8$ are the irreducible components, and the dimension of the variety is $5$. 
\end{proof}

\begin{theorem}\label{thm15}
    The variety of $3$-dimensional left-symmetric superalgebras with two-dimensional even part has dimension  $6$, and it has  $10$  irreducible components:
\begin{center}
        $\mathcal{C}_{1}=\overline{{\mathcal O}( {\mathrm{\hat{A}}^{k_1,k_2}_{1,3}} )},$ \quad
        $\mathcal{C}_{2}=\overline{{\mathcal O}( {\mathrm{\hat{A}}^{k_1,k_2}_{1,5}} )},$ \quad
        $\mathcal{C}_{3}=\overline{{\mathcal O}( {\mathrm{\hat{A}}_{1,6}} )},$ \quad
        $\mathcal{C}_{4}=\overline{{\mathcal O}( {\mathrm{\hat{A}}_{7_\lambda,1}} )},$ \quad
        $\mathcal{C}_{5}=\overline{{\mathcal O}( {\mathrm{\hat{A}}_{7_\lambda,2}^k} )},$ \\[4pt]
        $\mathcal{C}_{6}=\overline{{\mathcal O}( {\mathrm{\hat{A}}_{7_\lambda,3}^k} )},$ \quad
        $\mathcal{C}_{7}=\overline{{\mathcal O}( {\mathrm{\hat{A}}_{8,1}^k} )},$ \quad
        $\mathcal{C}_{8}=\overline{{\mathcal O}( {\mathrm{\hat{A}}_{10_\lambda,1}^k} )},$ \quad
        $\mathcal{C}_{9}=\overline{{\mathcal O}( {\mathrm{\hat{A}}_{10_\lambda,2}^k} )},$ \quad
        $\mathcal{C}_{10}=\overline{{\mathcal O}( {\mathrm{\hat{A}}_{10_\lambda,3}} )}.$
\end{center}
The components $\mathcal{C}_1$ and $\mathcal{C}_2$ have dimension $6$, and $\mathcal{C}_3,\ldots,\mathcal{C}_{10}$ have dimension $5$. In particular, there is only one rigid superalgebra, $\mathrm{\hat{A}}_{1,6}$.
\end{theorem}

\begin{proof}
The dimensions of the closures of the families are given in the following table:
\begin{longtable}{l|c||l|c}
\hline
$J$ & $\dim\overline{\mathcal{O}(J)}$ 
  & $J$ & $\dim\overline{\mathcal{O}(J)}$ \\
\hline
$\mathrm{\hat{A}}^{k_1,k_2}_{1,3}$ & 6 
  & $\mathrm{\hat{A}}^{k}_{7_\lambda,3}$ & 5 \\
$\mathrm{\hat{A}}^{k_1,k_2}_{1,5}$ & 6 
  & $\mathrm{\hat{A}}^{k}_{8,1}$ & 5 \\
$\mathrm{\hat{A}}_{1,6}$ & 5 
  & $\mathrm{\hat{A}}^{k}_{10_\lambda,1}$ & 5 \\
$\mathrm{\hat{A}}_{7_\lambda,1}$ & 5 
  & $\mathrm{\hat{A}}^{k}_{10_\lambda,2}$ & 5 \\
$\mathrm{\hat{A}}^{k}_{7_\lambda,2}$ & 5 
  & $\mathrm{\hat{A}}_{10_\lambda,3}$ & 5 \\
\hline
\end{longtable}
\noindent The irreducible closed sets $\mathcal{C}_1,\ldots,\mathcal{C}_{10}$ are pairwise distinct, since generic members of different families are not isomorphic (their even parts, which are $2$-dimensional LSAs, are different, or, for families with the same even part, they are distinguished by the products $fe_i$ and $ff$).

\medskip

The following degenerations show that all the other superalgebras of Theorem~\ref{thm12} belong to $\mathcal{C}_1\cup\cdots\cup\mathcal{C}_{10}$:

\begin{longtable}{lcl|lcl}
$\mathrm{ \hat{A}}_{1,3}^{k,0}$ & $\xrightarrow{\big(e_1+e_2, \ te_2, \ f \big)}$ & $\mathrm{ \hat{A}}_{2,2}^k$ &
$\mathrm{ \hat{A}}^{k}_{3,3} $ & $\xrightarrow{\big(e_1, \ t e_2, \ f \big)}$ & $\mathrm{ \hat{A}}_{3,2}^{k}$ \\[4pt]

$\mathrm{ \hat{A}}_{1,3}^{k-\frac{1}{t},\frac{1}{t}}$ & $\xrightarrow{\big( e_1+e_2, \ te_2, \ f \big)}$ & $\mathrm{ \hat{A}}_{2,3}^k$ &
$\mathrm{ \hat{A}}_{1,3}^{k,\frac{1}{t}}$ & $\xrightarrow{\big( e_1, \ te_2, \ f \big)}$ & $\mathrm{ \hat{A}}_{3,3}^k$ \\[4pt]

$\mathrm{ \hat{A}}^{k}_{2,6} $ & $\xrightarrow{\big(e_1+t^2 e_2, \ t e_2, \ f \big)}$ & $\mathrm{ \hat{A}}_{2,5}^k$ &
$\mathrm{ \hat{A}}_{1,6}$ & $\xrightarrow{\big(e_1, \ t^2e_2, \ tf \big)}$ & $\mathrm{ \hat{A}}_{3,4}$ \\[4pt]

$\mathrm{ \hat{A}}_{1,5}^{k-\frac{1}{t},\frac{1}{t}}$ & $\xrightarrow{\big(e_1+e_2, \ te_2, \ f \big)}$ & $\mathrm{ \hat{A}}_{2,6}^k$ &
$\mathrm{ \hat{A}}^{k}_{3,7} $ & $\xrightarrow{\big(e_1, \ t e_2, \ f \big)}$ & $\mathrm{ \hat{A}}_{3,6}^{k}$ \\[4pt]

$\mathrm{ \hat{A}}_{1,6}$ & $\xrightarrow{\big(e_1+e_2, \ t^2e_2, \ tf \big)}$ & $\mathrm{ \hat{A}}_{2,7}$ &
$\mathrm{ \hat{A}}_{1,5}^{\frac{1}{t},k}$ & $\xrightarrow{\big( e_2, \ te_1, \ f \big)}$ & $\mathrm{ \hat{A}}_{3,7}^k$ \\[4pt]

$\mathrm{ \hat{A}}_{1,3}^{1,\frac{1}{t}}$ & $\xrightarrow{\big(t e_2, \ t e_1, \ f \big)}$ & $\mathrm{ \hat{A}}_{4,1}$ &
$\mathrm{ \hat{A}}_{1,6}$ & $\xrightarrow{\big(e_2, \ te_1, \ f \big)}$ & $\mathrm{ \hat{A}}_{3,8}$ \\[4pt]

$\mathrm{ \hat{A}}_{1,6}$ & $\xrightarrow{\big(te_1, \ t^2e_2, \ tf \big)}$ & $\mathrm{ \hat{A}}_{4,2}$ &
$\mathrm{ \hat{A}}_{1,3}^{\frac{1}{2t^2},\frac{1}{2t^2}}$ & $\xrightarrow{\big(-t e_1+te_2, \ 2t^2 e_2, \ f \big)}$ & $\mathrm{ \hat{A}}_{5,4}$ \\[4pt]

$\mathrm{ \hat{A}}_{3,4} $ & $\xrightarrow{\big(-t^2 e_1+e_2, \ t^4 e_1, \ t f \big)}$ & $\mathrm{ \hat{A}}_{5,1}$ &
$\mathrm{ \hat{A}}_{5,4} $ & $\xrightarrow{\big(t e_1+t^2 e_2, \ t^2 e_2, \ f \big)}$ & $\mathrm{ \hat{A}}_{5,2}$ \\[4pt]

$\mathrm{ \hat{A}}_{10_1,1}^k$ & $\xrightarrow{\big(\frac{1}{t}e_1, \ e_2, \ f \big)}$ & $\mathrm{ \hat{A}}_{9,1}^k$ &
$\mathrm{ \hat{A}}_{10_1,2}^k$ & $\xrightarrow{\big(\frac{1}{t}e_1, \ e_2, \ f \big)}$ & $\mathrm{ \hat{A}}_{9,2}^k$ \\[4pt]

$\mathrm{ \hat{A}}_{5,4}$ & $\xrightarrow{\big(t e_1+e_2, \ t^2 e_2, \ f \big)}$ & $\mathrm{ \hat{A}}_{5,3}$ &
$\mathrm{ \hat{A}}_{10_1,3}$ & $\xrightarrow{\big(\frac{1}{t^2}e_1, \ e_2, \ \frac{1}{t}f \big)}$ & $\mathrm{ \hat{A}}_{9,3}$ \\[4pt]

$\mathrm{ \hat{A}}^k_{7_{t-1},3}$ & $\xrightarrow{\big(e_1, \ e_2-\frac{1}{t}e_1, \ f \big)}$ & $\mathrm{ \hat{A}}^k_{11,1}$ &
$\mathrm{ \hat{A}}^k_{7_{t-1},2}$ & $\xrightarrow{\big(e_1, \ e_2-\frac{1}{t}e_1, \ f \big)}$ & $\mathrm{ \hat{A}}^k_{11,3}$ \\[4pt]

$\mathrm{ \hat{A}}_{7_{t-1},1}$ & $\xrightarrow{\big(e_1, \ e_2-\frac{1}{t}e_1, \ f \big)}$ & $\mathrm{ \hat{A}}_{11,2}$ &
& & \\[4pt]
\end{longtable}

\noindent Here $\mathrm{\hat{A}}_{7_{\mu},i}^{k}$ and $\mathrm{\hat{A}}_{10_{\mu},i}^{k}$ denote the superalgebras $\mathrm{\hat{A}}_{7_\lambda,i}^{k}$ and $\mathrm{\hat{A}}_{10_\lambda,i}^{k}$ with $\lambda=\mu$. The degenerations $\mathrm{\hat{A}}_{1,3}^{k,0}\to\mathrm{\hat{A}}_{2,2}^k$, $\mathrm{\hat{A}}^{k}_{3,3}\to\mathrm{\hat{A}}_{3,2}^{k}$, $\mathrm{\hat{A}}^{k}_{2,6}\to\mathrm{\hat{A}}_{2,5}^k$ and $\mathrm{\hat{A}}^{k}_{3,7}\to\mathrm{\hat{A}}_{3,6}^{k}$ hold for every $k\in\mathbb{C}$ (including the values $k=0$ and $k=1$, which give the superalgebras $\mathrm{\hat{A}}_{2,1}$, $\mathrm{\hat{A}}_{3,1}$, $\mathrm{\hat{A}}_{2,4}$, $\mathrm{\hat{A}}_{3,5}$ of \cite{ZhangBai2012}), and the superalgebras $\mathrm{\hat{A}}_{7_{-1},i}$ belong to the closures of the corresponding families with even part $(\mathrm{A}_7)_\lambda$. Note that, for $\lambda=-1+t$, the basis $e_1,\ e_2-\frac{1}{t}e_1$ of the even part $(\mathrm{A}_7)_{\lambda}$ satisfies $(e_2-\frac{1}{t}e_1)e_1=-e_1$ and $(e_2-\frac{1}{t}e_1)(e_2-\frac{1}{t}e_1)=e_1+\lambda(e_2-\frac{1}{t}e_1)$, so that the $2$-dimensional LSA $(\mathrm{A}_7)_{-1+t}$ degenerates to $(\mathrm{A}_{11})$; consequently, the superalgebras with even part $(\mathrm{A}_{11})$ do not give irreducible components. 

\medskip

By the table of dimensions, $\mathcal{C}_i\not\subseteq\mathcal{C}_j$ whenever $i\neq j$ and $\dim\mathcal{C}_i\geq\dim\mathcal{C}_j$. Hence it is sufficient to prove that none of the superalgebras (families) defining $\mathcal{C}_3,\ldots,\mathcal{C}_{10}$ belongs to $\mathcal{C}_1$ or $\mathcal{C}_2$. This is done in the following table, where each non-degeneration is proved by Lemma~\ref{main2} with the set $\mathcal{W}$ given in the second column ($c_{ij}^k$ are the structure constants in the basis $x_1=e_1,\ x_2=e_2,\ x_3=f$; the non-degenerations whose target is a family are proved for generic values of its parameters):

{\renewcommand{\arraystretch}{1.4}
\begin{longtable}{|rcl|l|}
\hline
\multicolumn{3}{|c|}{\textbf{Non-degeneration}} & \textbf{Relations defining} $\mathcal{W}$ \\
\hline
$\mathrm{\hat{A}}^{k_1,k_2}_{1,3},\ \mathrm{\hat{A}}^{k_1,k_2}_{1,5}$ 
  & $\not\to$ 
  & $\mathrm{\hat{A}}_{1,6}$ 
  & $c_{33}^1=c_{33}^2=0$ \\
\hline
$\mathrm{\hat{A}}^{k_1,k_2}_{1,3},\ \mathrm{\hat{A}}^{k_1,k_2}_{1,5}$ 
  & $\not\to$ 
  & $\begin{array}{l}
       \mathrm{\hat{A}}_{7_\lambda,1},\ \mathrm{\hat{A}}^{k}_{7_\lambda,2},\ 
       \mathrm{\hat{A}}^{k}_{7_\lambda,3},\ \mathrm{\hat{A}}^{k}_{8,1},\\
       \mathrm{\hat{A}}^{k}_{10_\lambda,1},\ \mathrm{\hat{A}}^{k}_{10_\lambda,2},\ 
       \mathrm{\hat{A}}_{10_\lambda,3}
     \end{array}$
  & $c_{12}^1=c_{21}^1,\quad c_{12}^2=c_{21}^2$ \\
\hline
\end{longtable}

}

Consequently, $\mathcal{C}_1,\ldots,\mathcal{C}_{10}$ are the irreducible components, and the dimension of the variety is $6$.
\end{proof}

\end{document}